\documentclass[a4paper,10pt]{amsart}

\usepackage{amsmath,amssymb,amsthm,hyperref,graphicx,environ,dsfont,enumitem,tikz,float,tikz-cd,placeins,bm,color,mathrsfs,comment,stmaryrd,faktor}
\usepackage[normalem]{ulem}
\newtheorem{theorem}{Theorem}
\newtheorem{proposition}[theorem]{Proposition}
\newtheorem{corollary}[theorem]{Corollary}
\newtheorem{lemma}[theorem]{Lemma}
\newtheorem{conjecture}[theorem]{Conjecture}

\theoremstyle{definition}

\newtheorem{example}[theorem]{Example}
\newtheorem{definition}[theorem]{Definition}
\newtheorem{remark}[theorem]{Remark}

\allowdisplaybreaks

\newcommand{\dd}{\mathrm{d}}
\newcommand{\ii}{\mathrm{i}}
\newcommand{\ee}{\mathrm{e}}
\newcommand{\SU}{\mathrm{SU}}

\newcommand{\A}{\mathrm{A}}

\newcommand{\X}{\mathrm{X}}

\newcommand{\NN}{\mathbb N} 
\newcommand{\ZZ}{\mathbb Z} 
\newcommand{\QQ}{\mathbb Q} 
\newcommand{\CC}{\mathbb C} 
\DeclareMathOperator{\re}{Re} 
\DeclareMathOperator{\im}{Im} 

\DeclareMathOperator{\diag}{diag}

\DeclareMathOperator{\id}{id}

\DeclareMathOperator{\wt}{wt}
\DeclareMathOperator{\reg}{reg}
\DeclareMathOperator{\rev}{rev}
\DeclareMathOperator*{\res}{Res}
\DeclareMathOperator{\Li}{Li}

\begin{document}
\title[Trivial zeros of zeta functions of $\A_r$ type]{Trivial zeros of zeta functions of type $\mathrm{A}_r$}
\author{Simon Rutard}
\address{IMSc, CIT Campus, Taramani, Chennai 600113, India}
\email{simont@imsc.res.in}
\author{Takeshi Shinohara}
\address{Graduate School of Mathematics, Nagoya University, Furo-cho, Chikusa-ku, Nagoya 464-8602, Japan}
\email{shinohara.takeshi@math.nagoya-u.ac.jp}
\date{\today}
\subjclass[2020]{Primary 11M32}
\keywords{Euler--Zagier multiple zeta functions, Witten zeta function, zeta functions of root systems, trivial zeros, Eisenstein series.}
\begin{abstract}
	In this paper, we investigate the values of zeta functions of root systems at non-positive integer points. In particular, we focus on the zeta functions of type $\mathrm{A}_r$.
    Since non-positive integer points are included in the set of singularities of zeta functions of root systems, we introduce the \emph{ordered limit values} to define those values. 
    We study four particular ordered limit values and give explicit formulas for them, or recurrence formulas that allow us to calculate them inductively. 
    We then prove that three ordered limit values of zeta functions of type $\mathrm{A}_r$ vanish under some conditions that can be regarded as a multivariable analogue of the trivial zeros of the Riemann zeta function. 
    Finally, we discuss the other zeros of zeta functions of type $\mathrm{A}_r$ at non-positive integer points and a possible connection of such zeros and non-trivial identities among classical Eisenstein series.
\end{abstract}
\thanks{The first named author is supported by a Postdoctoral Fellowship from IMSc.}
\maketitle

\tableofcontents

\section{Introduction}
We begin by recalling the classical Riemann zeta function $\zeta(s):=\sum_{m\ge1}\frac{1}{m^s}$ ($\re(s)>1$). It is well-known that it has an analytic continuation to the whole complex plane and has a simple pole at $s=1$. It is also known that its values at non-positive integers are given explicitly by
\begin{align*}
    \zeta(-\ell) = (-1)^{\ell}\frac{B_{\ell+1}}{\ell+1} \quad (\ell\in\NN_0),
\end{align*}
where $B_{\ell}$ is the $\ell$-th Bernoulli number defined by the generating function
\begin{align*}
  \frac{X}{\ee^{X}-1} = \sum_{\ell\ge0}B_{\ell}\frac{X^\ell}{\ell!}.
\end{align*}
Since $B_\ell=0$ for odd $\ell>1$, we immediately obtain
\begin{align*}
    \zeta(-2\ell) = 0 \quad (\ell>0).
\end{align*}
Such zeros of the Riemann zeta function are called ``\emph{trivial zeros}'' today. 

Since Riemann's time, numerous generalizations and analogues of the Riemann zeta function have been introduced and studied.
For example, the Barnes zeta function, the Shintani zeta function, the Euler--Zagier multiple zeta function, the Witten zeta function, the zeta functions of root systems, multiple Dirichlet series with polynomial denominator, etc.
For many of these functions, analytic continuation has been established, and explicit formulas at non-positive integers are sometimes obtained.
In particular, Komori considered a very general multiple zeta function\footnote{Komori considers the more general Dirichlet series that has the twisted factor in the numerator.} that contains many classes of multivariable zeta functions: 
\begin{align*}
    \zeta({\bm{a}},\bm{b},\bm{s})
    := \sum_{m_1,\dots,m_R\ge0}
         \prod_{j=1}^N(a_j+b_{j1}m_1+\cdots+b_{jR}m_R)^{-s_j}
\end{align*}
and established the formula of its values at non-positive integer points:

\noindent
For $\bm{j}\in\NN_{0}^N$, ($\NN_0:=\ZZ_{\ge0}$) it follows 
\begin{align*}
 `\zeta(\bm{a},\bm{b},-\bm{j})'
 = (-1)^{j_1+\cdots+j_N}j_1!\cdots j_N!\frac{B_{w}(\bm{a},\bm{b};\bm{k})}{k_1!\cdots k_N!}
\end{align*} 
where $k_1,\dots,k_N$ depend on $j_1,\ldots,j_N$ and $B_{w}(\bm{a},\bm{b};\bm{k})$ is the \emph{generalized Bernoulli number}.
See \cite[Theorem 3.21]{komori2010integral} for details.
It should be mentioned here that, in most cases, the set of singularities of a multivariable zeta function contains the non-positive integer points and therefore the values at those points are not well-defined. It has become standard practice to consider these values as limit values and the symbol '$\zeta(\bm{a},\bm{b},-\bm{j})$' stands for a certain limit value.
In contrast to the case of the Riemann zeta function, since it is not so obvious when $B_{w}(\bm{a},\bm{b};\bm{k})$ vanishes, trivial zeros of the zeta function $\zeta({\bm{a}},\bm{b},\bm{s})$ have been less well understood. 
In \cite{essouabri20valueseulerzagier}, the authors pointed out that in a specific case, such as the Euler--Zagier case, the value $\zeta(\bm{a},\bm{b},-\bm{j})$ can be explicitly expressed by the finite sum of classical Bernoulli numbers.
However, it remains unclear when the value $\zeta(\bm{a},\bm{b},-\bm{j})$ vanishes.

One might wonder whether it would be more natural to consider a simpler class, rather than Komori’s very general class of multiple zeta functions, such as the Euler--Zagier multiple zeta functions 
\begin{align*}
    \zeta_{r}(s_1,\dots,s_r)
  := \sum_{m_1,\dots,m_r\ge1} \frac{1}{m_1^{s_1}(m_1+m_2)^{s_2}\cdots(m_1+\cdots+m_r)^{s_r}},
\end{align*}
and investigate their trivial zeros. 
In fact, relatively few results are known on trivial zeros even for special subclasses of Komori's multiple zeta functions, as well as for certain more general classes. 
The following works contain the main results known so far on trivial zeros of multiple zeta functions: \cite{zhao1999analyticcontiofmzf}, \cite{kamano2006lerchsformula}, and \cite{essouabri2021values}.
In the first two papers, the authors treated the Euler--Zagier multiple zeta function.
Zhao discussed the trivial zeros for Euler--Zagier double and triple zeta functions and Kamano showed that the \emph{center values} of the Euler--Zagier multiple zeta functions, namely certain limit values at non-positive integers, vanish at  $(-2\ell,\dots,-2\ell)$ ($\ell \in\NN$).
In \cite{essouabri2021values}, Essouabri and Matsumoto studied the multiple zeta functions whose defining multiple Dirichlet series have denominators given by polynomials, and discussed trivial zeros in the case of power sum denominators. Since their setting is somewhat different from ours, we will not discuss their results further in this paper.

There are also interesting results for one-variable (generalized) zeta functions. 
We mention the works of Romik \cite{romik2017representations} and Au \cite{au2024vanishingwitten}, as an example. 
They considered the function known as the Witten zeta function, which is defined as follows. 
\begin{align*}
  \zeta_{\mathfrak{g}}(s)
  := \sum_{\pi} (\dim \pi)^{-s}
\end{align*}
where the sum runs over all nonequivalent irreducible finite-dimensional complex representations of a semisimple Lie algebra $\mathfrak{g}$. 
Romik investigated the trivial zeros of the Witten zeta function in the case $\mathfrak{g}=\mathfrak{sl}_3$, while Au treated the more general cases. 
Even more interestingly, as we will explain later, trivial zeros of the Witten zeta function can sometimes be lifted to identities among classical Eisenstein series $G_{n}(\tau)$ ($n\ge3$, $\tau\in\mathbb{H}:=\{z\in\CC\,|\,\im{z}>0\}$). 
For example, Romik obtained the relation 
\begin{equation*}
 G_{6\ell+2}(\tau) 
 = \frac{1}{6\ell+1}\cdot\frac{(4\ell+1)!}{(2\ell)!^2}
   \sum_{k=1}^{\ell}
    \frac{\binom{2\ell}{2k-1}}{\binom{6\ell}{2\ell+2k-1}} G_{2\ell+2k}(\tau) G_{4\ell-2k+2}(\tau) 
 ,\quad (\ell\in\NN)
\end{equation*}
from the zeros of the Witten zeta function $\zeta_{\mathfrak{sl}_3}(-\ell)=0$, see \S \ref{sec: other zeros}. Au conjectured other relations between Eisenstein series coming from the vanishing of other Witten zeta functions in \cite{au2025valueswitten}.

\bigskip
We now state the purpose of the present paper. Our work may be viewed as a natural continuation of the studies of Zhao \cite{zhao1999analyticcontiofmzf}, Kamano \cite{kamano2006lerchsformula}, Romik \cite{romik2017representations}, and Au \cite{au2024vanishingwitten}. 
Namely, we investigate trivial zeros at non-positive integer points of the \emph{zeta functions of root systems} of type $\A_r$, given by 
\begin{equation}
 \zeta_{\A_r}(\bm{s})
  = \sum_{m_1,\dots,m_r\in\NN} \prod_{1\le i\le j\le r}(m_i+\cdots+m_j)^{-s_{ij}}. \label{eqn: explicit formula of Ar zeta function} 
\end{equation}
These functions form a class of multivariable complex functions that simultaneously generalizes both the Euler–Zagier multiple zeta functions and the Witten zeta functions.
Since when $r\ge2$, the set of all singularities of $\zeta_{\A_r}(\bm{s})$ contains the non-positive integer points, we consider these values at such points as a certain limit value.
In \S \ref{sec: val of zetaAr at negative}, we give a precise definition of the \emph{ordered limit values} of $\zeta_{\A_r}(\bm{s})$, denoted by $\zeta_{\A_r}(\overset{w}{-\bm{\ell}})$, where $\bm{\ell}$ is a non-negative integer tuple and $w$ is a permutation. 
Since it is not easy to calculate ordered limit values in general, in this paper we define and study four particular values, the \emph{Euler--Zagier values}, the \emph{regular values}, the \emph{diagonal values}, and the \emph{reverse values}. We will give some formulas for them that allow us to compute them inductively. 
We then show that these values vanish at certain points $-\bm{\ell} =(-\ell_{ij})_{1 \leq i \leq j \leq r}$, provided that the weight $\wt(\bm{\ell}) := \sum_{i \leq j} \ell_{ij}$ satisfies a certain congruence relation. 

\begin{theorem}\label{thm: main theorem}
    Let $\bm{\ell}=(\ell_{ij})_{1 \leq i \leq j \leq r} \in \NN_{0}^{\tfrac{r(r+1)}{2}}$ be a tuple such that \\
    \begin{equation}\label{eq:trivial_zero_condition}
       \ell_{11},\ell_{22},\ldots,\ell_{rr} \geq 1 \quad \text{and} \quad
       \wt(\bm{\ell}) \equiv r-1 \mod 2.
    \end{equation}\\
    Then for $w \in \{ \reg,\diag,\rev \}$, we have 
    \begin{align*}
        \zeta_{\A_r}(\overset{w}{-\bm{\ell}})=0.
    \end{align*}
\end{theorem}

\noindent
We refer to condition \eqref{eq:trivial_zero_condition} as the \emph{trivial zero condition} of $\zeta_{\A_r}(\bm{s})$.

When $r=1$, our theorem recovers the trivial zeros of the Riemann zeta function.
In this sense, the above Theorem can be regarded as a natural multivariable analogue of the trivial zeros of the Riemann zeta function. We also conjecture that the points where every ordered limits vanish correspond precisely to the points $\bm{\ell} \in \NN_{0}^{\frac{r(r+1)}{2}}$ satisfying the trivial zero condition \eqref{eq:trivial_zero_condition}.

\begin{conjecture}
    Let $\bm{\ell}=(\ell_{ij})_{1\le i\le j\le r}\in\NN_0^{\frac{r(r+1)}{2}}$. Then 
    \[
       \left(\forall w \in \mathfrak{S}_{\frac{r(r+1)}{2}}, \, \zeta_{\A_r}(\overset{w}{-\bm{\ell}}) = 0 \right) \Leftrightarrow \ell_{11},\ell_{22},\ldots,\ell_{rr} \geq 1 \text{ and } \wt(\bm{\ell}) = r-1 \mod 2.
    \]
\end{conjecture}

We also investigate the following point.
In \S \ref{sec: other zeros}, we will observe the other zeros of $\zeta_{\A_r}(\bm{s})$ at negative integers. In particular, although we have not succeeded in proving this, Mathematica calculations suggest that
\begin{align*}
    \zeta_{\A_r}\overset{\reg}{(-\ell,\dots,-\ell)} = 0 
\end{align*}
for specific $r$. Here, all components are the same negative integer $-\ell<0$ that does not satisfy the trivial zero condition \eqref{eq:trivial_zero_condition}. We will also discuss a potential connection between such zeros at non-positive integer points and some identities among Eisenstein series. We shall see that when $r=2$, we recover Romik's relation stated above. 

The paper is organized as follows. 
In Section 2, we recall the zeta functions of root systems as much as we need. 
In Section 3, we introduce the basic framework of the limit values of the zeta functions of $\A_r$ at non-positive integer points. We also define and study four particular values.
In Section 4, we shall prove Theorem \ref{thm: main theorem}.
Finally, in Section 5, we discuss the other zeros of zeta functions of $\A_r$ at non-positive integers. 


\section{Quick review of the zeta functions of root systems of type \texorpdfstring{$\A_r$}{Ar}}
\label{sec: quick review of zetaAr}
In this section, we briefly review the zeta functions of type $\A_r$.
In particular, we explain the \emph{matrix-like notation} for zeta functions of type $\A_r$ introduced in \cite{komiyama2025shuffle}.

Based on formula \eqref{eqn: explicit formula of Ar zeta function}, we rewrite $\zeta_{\A_r}(\bm{s})$ as 
\begin{align} \label{eqn: figure of variables of zeta functions of type Ar}
\zeta_{\A_r}(\bm{s})
= \zeta_{\A_r}\left(\begin{matrix}
               s_{11} &s_{12} & \cdots &s_{1r} \\
                       &s_{22} & \cdots &s_{2r} \\
                       &        & \ddots &\vdots \\
                       &       &      &s_{rr}
                 \end{matrix}\right).            
\end{align}

\begin{example}
For $s_1,s_2,s_3\in\CC$ with $\re(s_j)>1$,     
\begin{align*}
 \zeta_{\A_2}\left(\begin{matrix} s_1 &s_2 \\ &s_3 \end{matrix}\right)
 = \sum_{m_1,m_2\in\NN}\frac{1}{m_1^{s_1}(m_1+m_2)^{s_2} m_2^{s_3}}
 = \sum_{m_1,m_2\in\NN}\frac{1}{m_1^{s_3}(m_1+m_2)^{s_2} m_2^{s_1}}
 = \zeta_{\A_2}\left(\begin{matrix} s_3 &s_2 \\ &s_1 \end{matrix}\right).
\end{align*} 
For $s_1$, $s_2$, $s_3$, $s_4$, $s_5$, $s_6\in\CC$ with $\re(s_j)>1$, 
\begin{align*}
 \zeta_{\A_3}\left(\begin{smallmatrix} s_1 &s_2 &s_3 \\ &s_4 &s_5\\ & &s_6 \end{smallmatrix}\right)
 = \zeta_{\A_3}
     \left(\begin{smallmatrix} s_6 &s_5 &s_3 \\ &s_4 &s_2\\ & &s_1 \end{smallmatrix}\right),\quad
 \zeta_{\A_3}\left(\begin{smallmatrix} s_1 &0 &0 \\ &s_4 &s_5\\ & &s_6 \end{smallmatrix}\right)
 = \zeta_{\A_3}
     \left(\begin{smallmatrix} s_1 &0 &0 \\ &s_6 &s_5\\ & &s_4 \end{smallmatrix}\right).
\end{align*}
For $s_1,\dots,$ $s_{10}\in\CC$ with $\re(s_j)>1$,
\begin{align*}
\zeta_{\A_4}
    \left(\begin{smallmatrix} 
       s_1 &s_2 &s_3 &s_4 \\ &s_5 &s_6 &s_7\\ & &s_8 &s_9\\ & & &s_{10} \end{smallmatrix}\right)
 = \zeta_{\A_4}
    \left(\begin{smallmatrix} 
       s_{10} &s_9 &s_7 &s_4 \\ &s_8 &s_6 &s_3\\ & &s_5 &s_2\\ & & &s_1 \end{smallmatrix}\right),
\quad
  \zeta_{\A_4}
    \left(\begin{smallmatrix} 
       s_1 &0 &s_2 &s_3 \\ &s_4 &s_5 &s_6\\ & &s_7 &0\\ & & &s_8 \end{smallmatrix}\right)
 = \zeta_{\A_4}
    \left(\begin{smallmatrix} 
       s_1 &0 &s_2 &s_3 \\ &s_7 &s_5 &s_6\\ & &s_4 &0\\ & & &s_8 \end{smallmatrix}\right).
\end{align*}
\end{example}

From these examples, we sometimes observe symmetry in the variables $\bm{s}$.
In general, we have the following.

\begin{lemma}\label{lem: sym of variables}
    We have 
    \[
        \zeta_{\A_r}\left(
        \begin{smallmatrix}
            s_{11} & s_{12} & \cdots & s_{1r} \\
            & s_{22} & \cdots & s_{2r} \\
            & & \ddots & \vdots \\
            & & & s_{rr}
        \end{smallmatrix}
        \right) = \zeta_{\A_r}\left(
        \begin{smallmatrix}
            s_{rr} & s_{(r-1)r} & \cdots & s_{1r} \\
            & s_{(r-1)(r-1)} & \cdots & s_{1(r-1)} \\
            & & \ddots & \vdots \\
            & & & s_{11}
        \end{smallmatrix}
        \right).     
    \]
\end{lemma}

\begin{proof}
Putting $n_i:=m_{r+1-i}$ on the right-hand side of \eqref{eqn: explicit formula of Ar zeta function}, we have
\begin{align*}
\zeta_{\A_r}((s_{{ij}}))
  &= \sum_{n_1,\dots,n_r\in\NN} \prod_{1\le i\le j\le r}(n_{r+1-i}+\cdots+n_{r+1-j})^{-s_{ij}}\\
  &= \sum_{n_1,\dots,n_r\in\NN} \prod_{1\le i\le j\le r}(n_i+\cdots+n_{j})^{-s_{r+1-j,r+1-i}} 
   = \zeta_{\A_r}((s_{r+1-j,r+1-i}))
\end{align*}
as desired.
\end{proof}

\begin{remark}
The same idea holds for the more general case.
Let $1\le i< j\le r$. 
If $s_{k,\ell}=0$ for $1\le k\le i-1$ and $i\le \ell\le j-1$, 
$s_{p,q}=0$ for $i\le p\le j-1$ and $j\le q\le r$,
then, we have 
\begin{align*}
 \zeta_{\A_r}((s_{{k\ell}})) = \zeta_{\A_r}((s'_{k\ell}))
\end{align*}
where 
\begin{align*}
 s'_{k\ell} = \begin{cases}
             s_{i+j-\ell,i+j-k} &(i\le k\le j \ \ \text{and}\ \  k\le \ell\le j), \\
             s_{k\ell} &(\text{otherwise}).
           \end{cases}
\end{align*}
\end{remark}

\begin{remark}
    The previous lemma is a special case of a general symmetry property satisfied by zeta functions of root systems already discussed in \cite{komori2023root}. Let $\X_r$ be a root system, denote by $\zeta_{\X_r}(\bm{s})$ the zeta function of type $\X_r$, and let $\sigma$ be a graph automorphism of its Dynkin diagram. This automorphism acts on the tuple $\bm{s}=(s_{\alpha})_{\alpha \in \Delta^+}$ where $\Delta^+$ denotes the positive roots of $\X_r$. Then we have $\zeta_{\X_r}(\sigma \cdot \bm{s}) = \zeta_{\X_r}(\bm{s})$. In our setting here, $\X_r = \A_r$, and its only non-trivial Dynkin automorphism corresponds to $\sigma : \Pi_r \to \Pi_r, \ \alpha_j \mapsto \alpha_{r+1-j}$, where $\Pi_r$ denotes the set of fundamental roots of $\A_r$. The relation $\zeta_{\A_r}(\sigma \cdot \bm{s}) = \zeta_{\A_r}(\bm{s})$ induced by this non-trivial Dynkin automorphism corresponds to the relation shown in the previous lemma.
\end{remark}

To consider the values of $\zeta_{\A_r}(\bm{s})$ at non-positive integers, we need an analytic continuation to the larger region. 
The following {\it Mellin--Barnes integral formula} is useful for obtaining a recurrence formula for $\zeta_{\A_r}(\bm{s})$.

\begin{lemma}[{\cite[Corollary 14.5.1]{whittakermodernanalysis}}]\label{lem: MB-rec}
Let $s,\lambda\in\CC$ with $\re(s)>0$, $\lambda\ne0$, and $|\arg\lambda|<\pi$, we have 
\begin{equation}\label{eqn: MB formula}
    (1+\lambda)^{-s}=\frac{1}{2\pi \ii}\int_{(c)} \frac{\Gamma(s+z)\Gamma(-z)}{\Gamma(s)}\lambda^zdz
\end{equation}
where $\ii:=\sqrt{-1},$ $-\re(s)<c:=\re(z)<0$, and the symbol $(c)$ represents the path of integration along the vertical line from $c-\ii\infty$ to $c+ \ii\infty$.
\end{lemma}

We immediately find that $\zeta_{\A_r}$ has the following recurrence formula.

\begin{proposition}[{\cite[\S 6.2]{komori2010introduction}}]\label{prop: rec for Ar-zeta}
For $\bm{s}=(s_{ij})_{1\le i\le j\le r}\in\CC^{\frac{r(r+1)}{2}}$, we have 
{\small
\begin{equation}\label{eqn: rec rel for Ar zeta}
\zeta_{\A_r}
 \left(\begin{smallmatrix}
   s_{11} &s_{12} & \cdots &s_{1r} \\&s_{22} & \cdots &s_{2r} \\& & \ddots &\vdots \\ & & &s_{rr}
       \end{smallmatrix}\right) 
= \left(\frac{1}{2\pi \ii}\right)^{r-1}
  \!\!\int_{(\bm{c})} \!\!\!
    G(\bm{s};\bm{z}) \, \zeta_{\A_{r-1}}(\bm{s}'(\bm{z})) \zeta(s_{rr}-|\bm{z}|) d\bm{z}
\end{equation}}
where $(\bm{c})=(c_1)\times\cdots\times(c_{r-1})$, 
$c_i:=\re(z_i)<0$, $\re(s_{rr})-c_1-\cdots-c_{r-1}>1$, 
$|\bm{z}|:=\sum_{j=1}^{r-1} z_j$, 
$d\bm{z}=dz_1\cdots dz_{r-1}$, 
\begin{align*}
  \bm{s}'(\bm{z})=(s'_{ij}(\bm{z}))_{1\le i\le j\le r-1}
  := \begin{cases}
       s_{ij}              \quad &(1\le i\le j\le r-2),\\ 
       s_{ir-1}+s_{ir}+z_i \quad &(1\le i\le j=r-1),
     \end{cases}
\end{align*}
and 
\begin{align*}
 G(\bm{s};\bm{z}) := \prod_{i=1}^{r-1} \frac{\Gamma(s_{ir}+z_i)\Gamma(-z_i)}{\Gamma(s_{ir})}.
\end{align*}
\end{proposition}

\begin{proof}
For any $i=1,\dots,r-1$, we can write
\begin{align*}
 (m_i+\cdots+m_r)^{-s_{ir}} 
 = (m_i+\cdots+m_{r-1})^{-s_{ir}} \left(1+\frac{m_{r}}{m_i+\cdots+m_{r-1}}\right)^{-s_{ir}}.
\end{align*}
Thus, applying \eqref{eqn: MB formula} to the series expression of $\zeta_{\A_r}(\bm{s})$, we obtain \eqref{eqn: rec rel for Ar zeta}.
\end{proof}

\begin{remark}
The formula \eqref{eqn: rec rel for Ar zeta} gives us the meromorphic continuation of $\zeta_{\A_r}$. By induction, we find that all non-positive integer points are included in the set of its singularities. 
We can also see that those points are points of indeterminacy. 
In fact, we will see that the limit values at those points are different in general, see Example \ref{example: ordered limit values for A2}.
\end{remark}

\section{Definition of ordered limit values and four particular values}
\label{sec: val of zetaAr at negative}
In this section, we study the values of $\zeta_{\A_r}(\bm{s})$ at non-positive integers.
As we mentioned in the previous section, the non-positive integers belong to the set of (possible) singularities of $\zeta_{\A_r}(\bm{s})$.
Thus, we first define those values as a certain ordered limit and exhibit some examples for $\A_2$. 
In general, since we have $\frac{r(r+1)}{2}!$ ordered limit values for $\A_r$, we do not present examples for general $r$.
Nevertheless, Example \ref{example: ordered limit values for A2} already suggests some properties of ordered limit values. 
In addition, we also define the \emph{Euler--Zagier value} and establish their decomposition formula in terms of Euler--Zagier multiple zeta functions. 
In \S \ref{subsec: Regular values}, we study the \emph{regular values} and prove the recurrence formula, which will be used in the proof of our main Theorem. 
In \S \ref{subsec: Diagonal values and reverse values}, we study the \emph{diagonal values} and the \emph{reverse values}. In particular, we prove an explicit formula for the diagonal values that gives us their vanishing.

We define the \emph{ordered limit values} at non-positive integer points as follows.
First, note that for any $r\in\NN$, there is a bijection $\varphi$ between the following two sets:
\begin{align*}
  \varphi: \Big\{ (i,j)\in\NN^2\,|\,1\le i\le j\le r \Big\} \ \rightarrow\ 
  \left\{ n\in\NN \,|\, 1\le n\le \frac{r(r+1)}{2} \right\},
\end{align*}
\[
    (i,j) \mapsto (i-1)r - \frac{(i-1)i}{2}+j. 
\]
This bijection $\varphi$ orders the pairs $(i,j)$ lexicographically.
\begin{definition}
Let $r\in\NN$, $\bm{\ell}=(\ell_{ij})\in\NN_{0}^{\frac{r(r+1)}{2}}$, and $\mathfrak{S}_{\frac{r(r+1)}{2}}$ be the $\frac{r(r+1)}{2}$-th symmetric group. 
For $w\in \mathfrak{S}_{\frac{r(r+1)}{2}}$, consider a permutation 
$\bar{w}:\{ (i,j)\in\NN^2\,|\,1\le i\le j\le r \}\rightarrow\{ (i,j)\in\NN^2\,|\,1\le i\le j\le r \}$ by
\begin{align*}
    \bar{w}:= \varphi^{-1}\circ w \circ \varphi.
\end{align*}
We consider the $\id$-limit by 
\begin{align*}
    \lim_{\bm{s}\rightarrow-\bm{\ell}}\!\!{}^{\id}
    &:= \lim_{s_{11}\rightarrow-\ell_{11}}
        \lim_{s_{12}\rightarrow-\ell_{12}}\cdots
        \lim_{s_{rr}\rightarrow-\ell_{rr}},
\end{align*}
where we arrange $\lim_{s_{ij}\rightarrow-\ell_{ij}}$ in lexicographic order for $(i,j)$.
Then, we define the ordered limit value by
\begin{align*}
   \zeta_{\A_r}(\overset{w}{-\bm{\ell}})
   := \lim_{\bm{s}\rightarrow-\bm{\ell}}\!\!{}^{w}\, \zeta_{\A_r}(\bm{s}),
\end{align*}
where 
\begin{align*}
    \lim_{\bm{s}\rightarrow-\bm{\ell}}\!\!{}^{w}
    &:= \lim_{s_{\bar{w}((11))}\rightarrow-\ell_{\bar{w}((11))}}
        \lim_{s_{\bar{w}((12))}\rightarrow-\ell_{\bar{w}((12))}}\cdots
        \lim_{s_{\bar{w}((rr))}\rightarrow-\ell_{\bar{w}((rr))}}.
\end{align*}
where we also arrange $\lim_{s_{w(ij)}\rightarrow-\ell_{ij}}$ in lexicographic order for $(i,j)$
\end{definition}

\begin{example}\label{example: ordered limit values for A2}
We consider the case of $\A_2$. Note that we have a bijection 
\begin{align*}
  \varphi:\{ (1,1),\,(1,2),\,(2,2) \} \rightarrow \{ 1,\,2,\,3 \}, \quad
  (1,1) \mapsto 1, \ (1,2) \mapsto 2, \ (2,2) \mapsto 3
\end{align*}
and we have $6$ ordered limit values 
$\zeta_{\A_2}(\overset{w}{\begin{smallmatrix} -\ell_{11} & -\ell_{12} \\ &-\ell_{22} \end{smallmatrix}})$ 
for $w\in\mathfrak{S}_3$. 

\begin{itemize}[leftmargin=1em,label=$\bullet$]
\item
For $\id\in\mathfrak{S}_3$, the value $\zeta_{\A_2}(\overset{\id}{-\bm{\ell}})$ is given by
\begin{align*}
  \zeta_{\A_2}(\overset{\id}{-\bm{\ell}})
  =  \zeta_{\A_2}(\overset{\id}{\begin{smallmatrix} -\ell_{11} & -\ell_{12} \\ &-\ell_{22} \end{smallmatrix}})
  := \lim_{s_{11}\rightarrow-\ell_{11}}\lim_{s_{12}\rightarrow-\ell_{12}}\lim_{s_{22}\rightarrow-\ell_{22}}
      \zeta_{\A_2}(\begin{smallmatrix} s_{11} &s_{12} \\ &s_{22} \end{smallmatrix}).
\end{align*}
Since the series converges when $\re(s_{12}-\ell_{22})>1$, in that case we have 
{\small
\begin{align*}
   \lim_{s_{22}\rightarrow-\ell_{22}} 
   \zeta_{\A_2}(\begin{smallmatrix} s_{11} &s_{12} \\ &s_{22} \end{smallmatrix})
   = \sum_{k_{22}=0}^{\ell_{22}}\binom{\ell_{22}}{k_{22}}(-1)^{k_{22}}
      \sum_{m_1,m_2\in\NN}
         \frac{1}{m_1^{s_{11}-k_{22}}(m_1+m_2)^{s_{12}-\ell_{22}+k_{22}}}    .
\end{align*}}
Note that we use the binomial theorem 
\[
   m_2^{\ell_{22}} = \Big[ (m_1+m_2) - m_1 \Big]^{\ell_{22}} 
                = \sum_{k_{22}=0}^{\ell_{22}}
                    \binom{\ell_{22}}{k_{22}}(-1)^{k_{22}}m_1^{k_{22}}(m_1+m_2)^{\ell_{22}-k_{22}}
\]
and the last double infinite series is the Euler--Zagier double zeta function $\zeta_{2}(s_{11}-k_{22},s_{12}-\ell_{22}+k_{22})$. 
Thus, we have the decomposition 
\begin{equation}\label{eqn: decomp formula for zetaa2}
  \zeta_{\A_2}(\begin{smallmatrix} s_{11} &s_{12} \\ &-\ell_{22} \end{smallmatrix})
  = \sum_{k_{22}=0}^{\ell_{22}}\binom{\ell_{22}}{k_{22}}(-1)^{k_{22}}
      \zeta_{2}(s_{11}-k_{22},s_{12}-\ell_{22}+k_{22}).
\end{equation}
Akiyama, Egami, and Tanigawa \cite{akiyama2001analytic} defined the \emph{regular value} of the Euler--Zagier multiple zeta function as 
\begin{align*}
  \zeta_r^{\rm reg}(-\bm{\ell}) = \zeta_r^{\rm reg}(-\ell_1,\dots,-\ell_r)
 := \lim_{s_1\rightarrow -\ell_1}\cdots \lim_{s_r\rightarrow -\ell_r}\zeta_r(s_1,\dots,s_r).
\end{align*}
Therefore, the value $\zeta_{\A_2}(\overset{\id}{-\bm{\ell}})$ is given by the regular value of the Euler--Zagier double zeta function:
\begin{align*}
    \zeta_{\A_2}(\overset{\id}{-\bm{\ell}})
    = \sum_{k_{22}=0}^{\ell_{22}}\binom{\ell_{22}}{k_{22}}(-1)^{k_{22}}
        \zeta^{\reg}_{2}(-\ell_{11}-k_{22},-\ell_{12}-\ell_{22}+k_{22}).
\end{align*}

\item For $(13)\in\mathfrak{S}_3$, by Lemma \ref{lem: sym of variables}, we have
\begin{align*}
    \zeta_{\A_2}(\overset{(13)}{\begin{smallmatrix}-\ell_{11}&-\ell_{12}\\&-\ell_{22}\end{smallmatrix}}) 
    = \lim_{s_{22}\rightarrow-\ell_{22}}\lim_{s_{12}\rightarrow-\ell_{12}}\lim_{s_{11}\rightarrow-\ell_{11}}
         \zeta_{\A_2}(\begin{smallmatrix} s_{11} &s_{12} \\ &s_{22} \end{smallmatrix})
    =\zeta_{\A_2}(\overset{\id}{\begin{smallmatrix}-\ell_{22}&-\ell_{12}\\&-\ell_{11}\end{smallmatrix}}).
\end{align*}

\item For $(12)\in\mathfrak{S}_3$, Since the decomposition formula \eqref{eqn: decomp formula for zetaa2}
and the limit value 
\begin{align*}
  \zeta_r^{\rm rev}(-\bm{\ell}) = \zeta_r^{\rm rev}(-\ell_1,\dots,-\ell_r)
 := \lim_{s_r\rightarrow -\ell_r}\cdots \lim_{s_1\rightarrow -\ell_1}\zeta_r(s_1,\dots,s_r)
\end{align*}
is called \emph{reverse value} in \cite{akiyama2001analytic},
we have 
\begin{align*}
    \zeta_{\A_2}(\overset{(12)}{-\bm{\ell}})
    = \sum_{k_{22}=0}^{\ell_{22}}\binom{\ell_{22}}{k_{22}}(-1)^{k_{22}}
        \zeta^{\rev}_{2}(-\ell_{11}-k_{22},-\ell_{12}-\ell_{22}+k_{22}).
\end{align*}

\item For $(132)\in\mathfrak{S}_3$, by Lemma \ref{lem: sym of variables}, we have
\begin{align*}
    \zeta_{\A_2}(\overset{(132)}{\begin{smallmatrix}-\ell_{11}&-\ell_{12}\\&-\ell_{22}\end{smallmatrix}}) 
    = \lim_{s_{12}\rightarrow-\ell_{12}}\lim_{s_{22}\rightarrow-\ell_{22}}\lim_{s_{11}\rightarrow-\ell_{11}}
         \zeta_{\A_2}(\begin{smallmatrix} s_{11} &s_{12} \\ &s_{22} \end{smallmatrix}) 
    =\zeta_{\A_2}(\overset{(12)}{\begin{smallmatrix}-\ell_{22}&-\ell_{12}\\&-\ell_{11}\end{smallmatrix}}).
\end{align*}

\item For $(23)\in\mathfrak{S}_3$, since when $\re(s_{11}-\ell_{12})>1$ and $\re(s_{22}-\ell_{12})>1$, it follows
\begin{align*}
    \sum_{m_1,m_2\in\NN} \frac{(m_1+m_2)^{\ell_{12}}}{m_1^{s_{11}}m_2^{s_{22}}}
    = \sum_{k_{12}=0}^{\ell_{12}} \binom{\ell_{12}}{k_{12}}\zeta(s_{11}-\ell_{12}+k_{12})\zeta(s_{22}-k_{12}).
\end{align*}
Thus, we have
\begin{align*}
    \zeta_{\A_2}(\overset{(23)}{-\bm{\ell}})
    &= \lim_{s_{11}\rightarrow-\ell_{11}}\lim_{s_{22}\rightarrow-\ell_{22}}\lim_{s_{12}\rightarrow-\ell_{12}}
         \zeta_{\A_2}(\begin{smallmatrix} s_{11} &s_{12} \\ &s_{22} \end{smallmatrix}) \\
    &= \sum_{k_{12}=0}^{\ell_{12}} \binom{\ell_{12}}{k_{12}}\zeta(-\ell_{11}-\ell_{12}+k_{12})\zeta(-\ell_{22}-k_{12}).
\end{align*}

\item For $(123)\in\mathfrak{S}_3$, by Lemma \ref{lem: sym of variables}, we have
\begin{align*}
    \zeta_{\A_2}(\overset{(123)}{\begin{smallmatrix}-\ell_{11}&-\ell_{12}\\&-\ell_{22}\end{smallmatrix}}) 
    = \lim_{s_{22}\rightarrow-\ell_{22}}\lim_{s_{11}\rightarrow-\ell_{11}}\lim_{s_{12}\rightarrow-\ell_{12}}
         \zeta_{\A_2}(\begin{smallmatrix} s_{11} &s_{12} \\ &s_{22} \end{smallmatrix})
    =\zeta_{\A_2}(\overset{(23)}{\begin{smallmatrix}-\ell_{22}&-\ell_{12}\\&-\ell_{11}\end{smallmatrix}}).
\end{align*}
\end{itemize}
It follows immediately from these calculations that
\begin{align*}
    \zeta_{\A_2}(\overset{\id}{\begin{smallmatrix} -\ell & -\ell_{12} \\ &-\ell \end{smallmatrix}})
    = \zeta_{\A_2}(\overset{(13)}{\begin{smallmatrix} -\ell & -\ell_{12} \\ &-\ell \end{smallmatrix}}),
    \quad 
    \zeta_{\A_2}(\overset{(12)}{\begin{smallmatrix} -\ell & -\ell_{12} \\ &-\ell \end{smallmatrix}})
    = \zeta_{\A_2}(\overset{(132)}{\begin{smallmatrix} -\ell & -\ell_{12} \\ &-\ell \end{smallmatrix}})
        \quad (\ell,\ell_{12}\in\NN_0), 
\end{align*}
and $\zeta_{\A_2}(\overset{(123)}{-\bm{\ell}})=\zeta_{\A_2}(\overset{(23)}{-\bm{\ell}})$ for $\bm{\ell}\in\NN_0^3$. 
Moreover, since we know that (cf. \cite{akiyama2001analytic} and \cite{akiyama2001multiple})
\begin{align*}
    \zeta_2^{\reg}(0,0)=\frac{1}{3}, \quad \zeta_2^{\rev}(0,0)=\frac{5}{12},
\end{align*}
we have the following:
\begin{align*}
  \zeta_{\A_2}(\overset{\id}{\bm{0}})= \zeta_{\A_2}(\overset{(13)}{\bm{0}}) = \frac{1}{3}\ne 
  \zeta_{\A_2}(\overset{(12)}{\bm{0}})= \zeta_{\A_2}(\overset{(132)}{\bm{0}}) = \frac{5}{12}\ne
  \zeta_{\A_2}(\overset{(23)}{\bm{0}})= \zeta_{\A_2}(\overset{(123)}{\bm{0}}) = \frac{1}{4}.
\end{align*}
\end{example}

\begin{remark}
  For any permutation $w$ and any tuple $\bm{\ell}\in\NN_0^{\frac{r(r+1)}{2}}$, the corresponding ordered limit exists by a result of Komori \cite{komori2010integral}. It is natural to expect that some ordered limit values may coincide with each other, and thus their number of different values for a given non-positive integer tuple will be less than $\frac{r(r+1)}{2}!$. 
  We do not study this subject in the present paper, that is, the problem {\it ``when are two ordered limit values equal?''}, but we will investigate it in the forthcoming paper.
\end{remark}

For given $w\in\mathfrak{S}_{\frac{r(r+1)}{2}}$, it is not easy to calculate ordered limit values $\zeta_{\A_r}(\overset{w}{-\bm{\ell}})$ in general, but sometimes the calculation becomes much simpler. To see this, we define the four particular ordered limit values, namely the \emph{Euler--Zagier} values, the \emph{regular} values, the \emph{diagonal} values, and the \emph{reverse} values and study their explicit formula or some formulas that allow us to calculate them inductively.

\subsection{Euler--Zagier values}
In this subsection, we study the Euler--Zagier values, defined as follows.

\begin{definition}
Let $\bm{\ell}=(\ell_{ij})\in\NN_{0}^{\frac{r(r+1)}{2}}$. 
For a permutation $w=w_{EZ}$ of $\{1,\dots,r\}$ (clearly $w_{EZ}\in\mathfrak{S}_{\frac{r(r+1)}{2}}$), 
we call the value $\zeta_{\A_r}(\overset{w_{EZ}}{-\bm{\ell}})$ the \emph{Euler--Zagier value}.
In other words, it is defined as
\begin{align*}
\zeta_{\A_r}(\overset{w_{EZ}}{-\bm{\ell}}):= 
\lim_{\bm{s}\rightarrow-\bm{\ell}}\!\!{}^{w_{EZ}}\, \zeta_{\A_r}^{}(\bm{s})
&= \lim_{s_{\bar{w}(11)}\rightarrow-\ell_{\bar{w}(11)}}
     \cdots\lim_{s_{\bar{w}(1r)}\rightarrow-\ell_{\bar{w}(1r)}}
    \lim_{s_{22}\rightarrow-\ell_{22}}\cdots\lim_{s_{2r}\rightarrow-\ell_{2r}} \\
&\qquad \cdots \lim_{s_{r-1r-1}\rightarrow-\ell_{r-1r-1}}\lim_{s_{r-1\, r}\rightarrow-\ell_{r-1\, r}}
        \lim_{s_{rr}\rightarrow-\ell_{rr}}
          \zeta_{\A_r}^{}(\bm{s}).
\end{align*}
Here, we put $\bar{w}:=\bar{w}_{EZ}$ to save space.
\end{definition}

By definition, we can write
\begin{align*}
\zeta_{\A_r}(\overset{w_{EZ}}{-\bm{\ell}})
&= \lim_{s_{\bar{w}(11)}\rightarrow-\ell_{\bar{w}(11)}}
     \cdots\lim_{s_{\bar{w}(1r)}\rightarrow-\ell_{\bar{w}(1r)}}
   \zeta_{\A_r}
     \left(\begin{smallmatrix}
          s_{11} & s_{12}  & \cdots & s_{1r}  \\
                 & -\ell_{22} & \cdots & -\ell_{2r} \\
                 &         & \ddots & \vdots  \\
                 &         & & -\ell_{rr} 
         \end{smallmatrix}\right)
\end{align*}
and note that for $\zeta_{\A_r}
     \left(\begin{smallmatrix}
          s_{11} & s_{12}  & \cdots & s_{1r}  \\
                 & -\ell_{22} & \cdots & -\ell_{2r} \\
                 &         & \ddots & \vdots  \\
                 &         & & -\ell_{rr} 
         \end{smallmatrix}\right)$, 
the series expression is still valid when $\re(s_{1j})$ ($j=1,\dots,r$) are sufficiently large. 
We also find the following decomposition.         

\begin{lemma}
Let $\bm{\ell}'=(\ell_{ij})_{2\le i\le j\le r}\in\NN_0^{r(r-1)/2}$ and $s_{1j}\in\CC$ except for the singularities.
We have the decomposition
\begin{align*}
 \zeta_{\A_r}
   \left(\begin{smallmatrix}
          s_{11} & s_{12}  & \cdots & s_{1r}  \\
                 & -\ell_{22} & \cdots & -\ell_{2r} \\
                 &         & \ddots & \vdots  \\
                 &         & & -\ell_{rr} 
         \end{smallmatrix}\right)
 = \prod_{2\le i\le j\le r} \sum_{p_{ij}+q_{ij}=\ell_{ij}}  
     \binom{\ell_{ij}}{q_{ij}} (-1)^{q_{ij}} \zeta_{r}(\bm{s}(\bm{p},\bm{q})),
\end{align*}
where $\zeta_{r}(s_1\dots,s_r)$ is the Euler--Zagier multiple zeta function and
\begin{align*}
&\bm{s}(\bm{p},\bm{q})
= \left(  s_{1k}-\sum_{i=2}^{k}p_{ik}-\sum_{j=k+1}^{r}q_{k+1\,j}  \right)_{1\le k\le r}.
\end{align*}
We consider the empty sum as $0$.
\end{lemma}

\begin{proof}
Use the formula
\begin{align*}
   (m_i+\cdots+m_j)^{\ell_{ij}}
    &= \Big[ (m_1+\cdots+m_j) - (m_1+\cdots+m_{i-1}) \Big]^{\ell_{ij}}\\
    &= \sum_{p_{ij}+q_{ij}=\ell_{ij}} \binom{\ell_{ij}}{q_{ij}}(-1)^{q_{ij}} 
           (m_1+\cdots+m_{j})^{p_{ij}} (m_1+\cdots+m_{i-1})^{q_{ij}}
\end{align*}
and we obtain the equation for $\re(s_{1j}+\cdots+s_{1r})>\wt(\bm{\ell}')+r-j+1$ for $1\le j\le r$. 
We know that the functions on both sides have analytic continuations, and this completes the proof. 
\end{proof}

\begin{remark}
For a permutation $w$ of $\{1,\dots,r\}$, the limit values of Euler--Zagier multiple zeta function
\begin{align*}
    \lim_{s_{w(1)}\rightarrow-\ell_{w(1)}}\cdots\lim_{s_{w(r)}\rightarrow-\ell_{w(r)}}
    \zeta_r(s_1,\dots,s_r)
\end{align*}
have already been studied by several authors \cite{akiyama2001multiple}, \cite{sasaki2009multiple}, \cite{komori2010integral}, etc. 
Thus, we have  
\begin{align*}
  \zeta_{\A_r}\overset{w_{EZ}}{(-\bm{\ell})}
  = 
      \prod_{2\le i\le j\le r} \sum_{p_{ij}+q_{ij}=\ell_{ij}}  
        \binom{\ell_{ij}}{q_{ij}} (-1)^{q_{ij}} \zeta_{r}( \overset{w_{EZ}}{\bm{s}(\bm{p},\bm{q})} )
\end{align*}
and we may say that the Euler--Zagier values are well-understood.
\end{remark}

\subsection{Regular values}
\label{subsec: Regular values}
In this subsection, we study the \emph{regular} values of $\zeta_{\A_r}(\bm{s})$ at non-positive integer points. 
We prove the recurrence formula (see Proposition \ref{prop: recurrence formula for reg val of Ar-zeta}) which allows us to calculate them inductively.

We define the \emph{regular values} using a modified ordered limit.
\begin{definition}
Let $\bm{\ell}=(\ell_{ij})\in\NN_{0}^{\frac{r(r+1)}{2}}$. 
Consider the bijection $\bar{\varphi}:\{ (i,j)\in\NN^2\,|\,1\le i\le j\le r \}\rightarrow\{ n\in\NN \,|\, 1\le n\le \frac{r(r+1)}{2} \}$ defined by 
\begin{align*}
\bar{\varphi}: (i,j) \mapsto \frac{(j-1)j}{2}+i.
\end{align*}
For $\reg := \bar{\varphi}^{-1} \circ \id \circ \bar{\varphi}$ (abuse of notation), we set
\begin{align*}
  \lim_{\bm{s}\rightarrow-\bm{\ell}}\!\!{}^{\reg} 
  := \lim_{s_{\reg((11))}\rightarrow-\ell_{\reg((11))}}
     \lim_{s_{\reg((12))}\rightarrow-\ell_{\reg((12))}}\cdots
     \lim_{s_{\reg((rr))}\rightarrow-\ell_{\reg((rr))}}.
\end{align*}
We define the \emph{regular value} as 
\begin{align*}
    \zeta_{\A_r}(\overset{\reg}{-\bm{\ell}})
    := \lim_{\bm{s}\rightarrow-\bm{\ell}}\!\!{}^{\reg}\,
         \zeta_{\A_r}(\bm{s}).
\end{align*}
\end{definition}

\begin{example}
Let $\bm{\ell}\in\NN_0^{3}$. By definition, we have 
\begin{align*}
  \zeta_{\A_2}\overset{\reg}{ (\begin{smallmatrix} -\ell_{11} & -\ell_{12} \\ & -\ell_{22} \end{smallmatrix}) }
  = \lim_{s_{11}\rightarrow-\ell_{11}}\lim_{s_{12}\rightarrow-\ell_{12}}\lim_{s_{22}\rightarrow-\ell_{22}}
      \zeta_{\A_2}(\begin{smallmatrix} s_{11} & s_{12} \\ & s_{22} \end{smallmatrix}).
\end{align*}
Thus, for $\A_2$, we have $\zeta_{\A_2}(\overset{\reg}{-\bm{\ell}})=\zeta_{\A_2}(\overset{\id}{-\bm{\ell}})$. 
\end{example}

Now, we prove the recurrence formula for the regular values.

\begin{proposition}\label{prop: recurrence formula for reg val of Ar-zeta}
For $r\ge3$ and $\bm{\ell}\in\NN_0^{\frac{r(r+1)}{2}}$, we have 
\begin{equation}\label{eqn: rec. rel. for reg Ar zeta}\begin{split}
\zeta_{\A_r}(\overset{\reg}{-\bm{\ell}})
= &\prod_{j=1}^{r-1}\ 
   \sum_{p_j+q_j=\ell_{jr}}
   \binom{\ell_{jr}}{q_j}
     \zeta_{\A_{r-1}}(\overset{\reg}{-\bm{\ell}'(\bm{p})})
     \zeta(-\ell_{rr}-|\bm{q}|) \\
  & + \prod_{j=2}^{r-1}
      \sum_{p_j+q_j=\ell_{jr}}
        \binom{\ell_{jr}}{q_j}
        \frac{ (-1)^{\ell_{rr}+|\bm{q}'|+1} }{\ell_{1r}+\ell_{rr}+|\bm{q}'|+1}
        \binom{\ell_{1r}+\ell_{rr}+|\bm{q}'|}{\ell_{1r}}^{-1}
        \zeta_{\A_{r-1}}(\overset{\reg}{-\bm{\ell}''(\bm{p})})
\end{split}\end{equation}
where 
\begin{align*}
-\bm{\ell}'(\bm{p}) = (-\ell'_{ij}(\bm{p}))_{1\le i\le j\le r-1}
= \begin{cases}
    -\ell_{ij} \quad &(1\le i\le j\le r-2), \\
    -\ell_{i\,r-1}-p_i \quad &(1\le i\le r-1),
  \end{cases}
\end{align*}
$|\bm{q}|:=q_1+\cdots+q_{r-1}$, $|\bm{q}'|:=q_2+\cdots+q_{r-1}$, and 
\begin{align*}
-\bm{\ell}''(\bm{p}) = (-\ell''_{ij}(\bm{p}))_{1\le i\le j\le r-1}
= \begin{cases}
    -\ell_{ij} \quad &(1\le i\le j\le r-2), \\ 
    -\ell_{1\,r-1}-\ell_{1r}-\ell_{rr}-|\bm{q}'|-1 \quad &(i=1,\,j=r-1), \\
    -\ell_{i\,r-1}-p_i \quad &(2\le i\le r-1).
  \end{cases}
\end{align*} 
\end{proposition}

\begin{proof}
When $\re(s_{1r})>1+\sum_{j=2}^{r-1}\ell_{jr}$, we have
\begin{align*}
\zeta_{\A_r}\left(\begin{smallmatrix}
                      s_{11} &s_{12}  &\cdots &s_{1\,r-1} &s_{1r} \\
                             &s_{22}  &\cdots &s_{2\,r-1}&-\ell_{2r} \\
                             &        & \ddots &\vdots &\vdots \\
                             &        &        &s_{r-1\,r-1}   &-\ell_{r-1r} \\
                             &        &        &       &-\ell_{rr}
                      \end{smallmatrix}\right) 
= \prod_{j=2}^{r-1}\ 
   \sum_{p_j+q_j=\ell_{jr}}
   \binom{\ell_{jr}}{p_{j}}
   \zeta_{\A_r}
     \left(\begin{smallmatrix}
             s_{11} &s_{12}  &\cdots &s_{1\,r-1}   &s_{1r} \\
                    &s_{22}  &\cdots &s_{2\,r-1}-p_2   &0 \\
                    &        &\ddots &\vdots       &\vdots \\
                    &        &       &s_{r-1\,r-1}-p_{r-1} &0 \\
                    &        &       &             &-\ell_{rr}-|\bm{q'}|
           \end{smallmatrix}\right)
\end{align*}
where $|\bm{q'}|:=q_2+\cdots+q_{r-1}$.
By the Mellin--Barnes integral formula \eqref{eqn: MB formula}, we have 
\begin{align*}
 &\zeta_{\A_r}
     \left(\begin{smallmatrix}
             s_{11} &s_{12}  &\cdots &s_{1\,r-1}   &s_{1r} \\
                    &s_{22}  &\cdots &s_{2\,r-1}-p_2   &0 \\
                    &        &\ddots &\vdots       &\vdots \\
                    &        &       &s_{r-1\,r-1}-p_{r-1} &0 \\
                    &        &       &             &-\ell_{rr}-|\bm{q'}|
           \end{smallmatrix}\right) \\
&= \frac{1}{2\pi \ii}\int_{(c)}
   \frac{\Gamma(s_{1r}+z)\Gamma(-z)}{\Gamma(s_{1r})}
   \zeta_{\A_{r-1}}
     \left(\begin{smallmatrix}
             s_{11} &s_{12}  &\cdots &s_{1\,r-1}+s_{1r}+z   \\
                    &s_{22}  &\cdots &s_{2\,r-1}-p_2        \\
                    &        &\ddots &\vdots                \\
                    &        &       &s_{r-1\,r-1}-p_{r-1}  \\
                    &        &       &                      
           \end{smallmatrix}\right) 
   \zeta(-\ell_{rr}-|\bm{q'}|-z)dz.
\end{align*}
Here, $c:=\re(z)<-\ell_{rr}-|\bm{q'}|-1$. 
By shifting the integration to the right, we obtain 
\begin{align*}
 &\zeta_{\A_r}
     \left(\begin{smallmatrix}
             s_{11} &s_{12}  &\cdots &s_{1\,r-1}   &s_{1r} \\
                    &s_{22}  &\cdots &s_{2\,r-1}-p_2   &0 \\
                    &        &\ddots &\vdots       &\vdots \\
                    &        &       &s_{r-1\,r-1}-p_{r-1} &0 \\
                    &        &       &             &-\ell_{rr}-|\bm{q'}|
           \end{smallmatrix}\right) \\
&= \sum_{k=0}^{N}
   \binom{-s_{1r}}{k}
   \zeta_{\A_{r-1}}
     \left(\begin{smallmatrix}
             s_{11} &s_{12}  &\cdots &s_{1\,r-1}+s_{1r}+k   \\
                    &s_{22}  &\cdots &s_{2\,r-1}-p_2        \\
                    &        &\ddots &\vdots                \\
                    &        &       &s_{r-1\,r-1}-p_{r-1}  \\
                    &        &       &                      
           \end{smallmatrix}\right) 
   \zeta(-\ell_{rr}-|\bm{q'}|-k) \\
&\quad 
 + \frac{\Gamma(s_{1r}-\ell_{rr}-|\bm{q'}|-1)\Gamma(\ell_{rr}+|\bm{q'}|+1)}
        {\Gamma(s_{1r})} 
   \zeta_{\A_{r-1}}
     \left(\begin{smallmatrix}
             s_{11} &s_{12}  &\cdots &s_{1\,r-1}+s_{1r}-\ell_{rr}-|\bm{q'}|-1   \\
                    &s_{22}  &\cdots &s_{2\,r-1}-p_2        \\
                    &        &\ddots &\vdots                \\
                    &        &       &s_{r-1\,r-1}-p_{r-1}  \\
                    &        &       &                      
           \end{smallmatrix}\right) \\
& + I_{N+\epsilon}(\bm{s},\bm{p},\bm{q}),  
\end{align*}
where $N\in\NN$, $\epsilon>0$ (sufficiently small), and 
\begin{align*}
I_{N+\epsilon}(\bm{s},\bm{p},\bm{q})
= \frac{1}{2\pi \ii}\int_{(N+\epsilon)}
   \frac{\Gamma(s_{1r}+z)\Gamma(-z)}{\Gamma(s_{1r})}
   \zeta_{\A_{r-1}}
     \left(\begin{smallmatrix}
             s_{11} &s_{12}  &\cdots &s_{1\,r-1}+s_{1r}+z   \\
                    &s_{22}  &\cdots &s_{2\,r-1}-p_2        \\
                    &        &\ddots &\vdots                \\
                    &        &       &s_{r-1\,r-1}-p_{r-1}  \\
                    &        &       &                      
           \end{smallmatrix}\right) 
   \zeta(-\ell_{rr}-|\bm{q'}|-z)dz.
\end{align*}
Since $\lim_{s_{1r}\rightarrow-\ell_{1r}}I_{N+\epsilon}(\bm{s},\bm{p},\bm{q})=0$ for $\ell_{1r}\in\NN_0$, we set $N=\ell_{1r}$ and get  
\begin{align*}
&\lim_{s_{1r}\rightarrow-\ell_{1r}}
\zeta_{\A_r}
     \left(\begin{smallmatrix}
             s_{11} &s_{12}  &\cdots &s_{1\,r-1}   &s_{1r} \\
                    &s_{22}  &\cdots &s_{2\,r-1}-p_2   &0 \\
                    &        &\ddots &\vdots       &\vdots \\
                    &        &       &s_{r-1\,r-1}-p_{r-1} &0 \\
                    &        &       &             &-\ell_{rr}-|\bm{q'}|
           \end{smallmatrix}\right) \\
&= \sum_{p_1+q_1=\ell_{1r}}
   \binom{\ell_{1r}}{p_1}
   \zeta_{\A_{r-1}}
     \left(\begin{smallmatrix}
             s_{11} &s_{12}  &\cdots &s_{1\,r-1}-p_1   \\
                    &s_{22}  &\cdots &s_{2\,r-1}-p_2        \\
                    &        &\ddots &\vdots                \\
                    &        &       &s_{r-1\,r-1}-p_{r-1}  \\
                    &        &       &                      
           \end{smallmatrix}\right) 
   \zeta(-\ell_{rr}-|\bm{q}|) \\
&\quad 
 + \frac{ (-1)^{\ell_{rr}+|\bm{q}|+1} }{\ell_{1r}+\ell_{rr}+|\bm{q}|+1}
   \binom{\ell_{1r}+\ell_{rr}+|\bm{q}|}{\ell_{1r}}^{-1} 
   \zeta_{\A_{r-1}}
     \left(\begin{smallmatrix}
             s_{11} &s_{12}  &\cdots &s_{1\,r-1}-\ell_{1r}-\ell_{rr}-|\bm{q'}|-1   \\
                    &s_{22}  &\cdots &s_{2\,r-1}-p_2        \\
                    &        &\ddots &\vdots                \\
                    &        &       &s_{r-1\,r-1}-p_{r-1}  \\
                    &        &       &                      
           \end{smallmatrix}\right)
\end{align*}
Thus, we obtain the claim by taking the remaining limits.
\end{proof}

Here, we give a recurrence formula for $A_2$.

\begin{example}
By the same method of \eqref{eqn: rec. rel. for reg Ar zeta}, we have
\begin{equation}\label{eqn: example of reg rec rel for A2}\begin{split}
 \zeta_{\A_2} 
   \overset{\reg}{\left(\begin{smallmatrix}-\ell_{11} &-\ell_{12} \\ &-\ell_{22} \end{smallmatrix}\right)}
 = &\sum_{p_{12}+q_{12}=\ell_{12} }
   \binom{\ell_{12}}{q_{12}}
    \zeta(-\ell_{11}-p_{12})\zeta(-\ell_{22}-q_{12}) \\
    &+ \frac{(-1)^{\ell_{22}+1}}{\ell_{12}+\ell_{22}+1}
      \binom{\ell_{12}+\ell_{22}}{\ell_{22}}^{-1}
      \zeta(-\ell_{11}-\ell_{12}-\ell_{22}-1).
\end{split}\end{equation}
One can prove the vanishing of $\zeta_{\A_2}\overset{\reg}{(-\bm{\ell})}$ directly from this recurrence relation. 
Indeed, we find that the right-hand side vanishes when $\ell_{11},\ell_{22}>0$ and $\ell_{11}+\ell_{12}+\ell_{22}$ is odd. 
\end{example}

\begin{remark}
By \eqref{eqn: rec. rel. for reg Ar zeta}, one can see that $\zeta_{\A_r}\overset{\reg}{(-\bm{\ell})}$ coincides with the regular value of the Euler--Zagier multiple zeta function $\zeta^{\rm reg}_{r}(-\bm{\ell})$ when
\begin{align*}
 -\bm{\ell}=(-\ell_{ij})_{1\le i\le j\le r}
 =\begin{cases}
    -\ell_j \quad &(i=1,\ 1\le j\le r), \\
    0 \quad &(otherwise).
  \end{cases}
\end{align*}
In this case, equation \eqref{eqn: rec. rel. for reg Ar zeta} recovers the recurrence formula of the Euler--Zagier multiple zeta functions.
\end{remark}

\subsection{Diagonal values and reverse values}
\label{subsec: Diagonal values and reverse values}
In this section, we study other two types of ordered limit values, the \emph{diagonal} values and the \emph{reverse} values.
We will give an explicit formula for the diagonal values (Proposition \ref{prop: explicit formula for diag val}) and prove their vanishing (Theorem \ref{thm: vanishing of diag val}).
We then study a decomposition formula of the reverse values in terms of Euler--Zagier reverse values (Proposition \ref{prop: explicit formula for rev val}).

\begin{definition}
For $\bm{\ell}\in\NN_0^{\frac{r(r+1)}{2}}$, we define the \emph{diagonal value} of the zeta functions of $\A_r$ by
\begin{align*}
 \zeta_{\A_r}\overset{\diag}{(-\bm{\ell})}
 := \lim_{s_{ii}\rightarrow-\ell_{ii}} \zeta_{\A_r}(\bm{s};-\bm{\ell})
 := \lim_{s_{ii}\rightarrow-\ell_{ii}}
    \zeta_{\A_r}\left(\begin{smallmatrix}
                      s_{11} &-\ell_{12} &-\ell_{13} &\cdots &-\ell_{1r} \\
                             &s_{22}  &-\ell_{23} &\cdots &-\ell_{2r} \\
                             &        & \ddots &\ddots &\vdots \\
                             &        &        &s_{r-1\,r-1}   &-\ell_{r-1r} \\
                             &        &        &       &s_{rr}
                      \end{smallmatrix}\right).
\end{align*}

We also define the \emph{reverse value} of the zeta functions of $\A_r$ by
\begin{align*}
 \zeta_{\A_r}\overset{\rev}{(-\bm{\ell})}
 := \lim_{s_{1r}\rightarrow-\ell_{1r}} \zeta_{\A_r}(-\bm{\ell};s_{1r})
 := \lim_{s_{1r}\rightarrow-\ell_{1r}}
    \zeta_{\A_r}\left(\begin{smallmatrix}
                      -\ell_{11} &-\ell_{12} &-\ell_{13} &\cdots &s_{1r} \\
                             &-\ell_{22}  &-\ell_{23} &\cdots &-\ell_{2r} \\
                             &        & \ddots &\ddots &\vdots \\
                             &        &        &-\ell_{r-1\,r-1}   &-\ell_{r-1r} \\
                             &        &        &       &-\ell_{rr}
                      \end{smallmatrix}\right).
\end{align*}
\end{definition}

\begin{remark}
Note that if $\re(s_{11}),\dots,\re(s_{rr})$ are sufficiently large, the series expression \eqref{eqn: explicit formula of Ar zeta function} of 
$\zeta_{\A_r}\left(\begin{smallmatrix}
                      s_{11} &-\ell_{12} &-\ell_{13} &\cdots &-\ell_{1r} \\
                             &s_{22}  &-\ell_{23} &\cdots &-\ell_{2r} \\
                             &        & \ddots &\ddots &\vdots \\
                             &        &        &s_{r-1\,r-1}   &-\ell_{r-1r} \\
                             &        &        &       &s_{rr}
                      \end{smallmatrix}\right)$ 
is still valid. Note that the series expression also holds for $\zeta_{\A_r}(-\bm{\ell};s_{1r})$ when $\re(s_{1r})$ is large.
By Lemma \ref{decomp formula for the reg zetaAr}, the diagonal values do not depend on the order of limits $s_{ii}\rightarrow-\ell_{ii}$ for $i=1,\dots,r$.
\end{remark} 

\begin{lemma}\label{decomp formula for the reg zetaAr}
For $\bm{\ell}\in\NN_0^{r(r-1)/2}$ and $s_{ii}\in\CC$ except for their singularities, we have
\begin{align*}
 \zeta_{\A_r}(\bm{s};-\bm{\ell})
 = \prod_{1\le i<j \le r}
   \sum_{\substack{ k^{(ij)}_i+\cdots+k^{(ij)}_{j}=\ell_{ij} \\ \quad k^{(ij)}_p\ge0 }} \!\!\!\!
   \binom{\ell_{ij}}{k^{(ij)}_i,\dots,k^{(ij)}_{j}} 
   \prod_{i=1}^r\zeta(s_{ii}-K_i) 
\end{align*}
where $\binom{\ell_{ij}}{k^{(ij)}_i,\dots,k^{(ij)}_{j}}=\frac{\ell_{ij}!}{k^{(ij)}_i!\cdots k^{(ij)}_{j}!}$ are the multinomial coefficients and $K_i$ ($i=1,\dots,r$) are given by \eqref{eqn: definition of Ki} below.
\end{lemma}

\begin{proof}
At first, when $\re(s_{ii})$ is sufficiently large for each $i=1,\dots,r$, $\zeta_{\A_r}(\bm{s};-\bm{\ell})$ converges absolutely. 
By the multinomial theorem, we have
\begin{align*}
 \zeta_{\A_r}(\bm{s};-\bm{\ell})
 &= \sum_{m_1,\dots,m_r\in\NN} 
    \prod_{i=1}^rm_i^{-s_{ii}}
    \prod_{1\le i< j\le r}(m_i+\cdots+m_j)^{\ell_{ij}} \\
 &= \prod_{1 \leq i < j \leq r} \sum_{\substack{ k^{(ij)}_i+\cdots+k^{(ij)}_{j}=\ell_{ij} \\[1mm] k^{(ij)}_p\ge0 }}
      \binom{\ell_{ij}}{k^{(ij)}_i,\dots,k^{(ij)}_{j}}
     \sum_{m_1,\dots,m_r\in\NN} 
    \prod_{i=1}^rm_i^{K_i-s_{ii}}.
\end{align*}
Here, we put 
\begin{align}\label{eqn: definition of Ki}
 K_i := \sum_{\substack{1\le p\le i\le q\le r \\ p\ne q}} k_i^{(p\,q)}.
\end{align}
Thus, we obtain the equality in the statement, and we know that the functions on both sides have analytic continuations. Hence the claim.
\end{proof}

As a consequence, we obtain the explicit formula for the diagonal values.

\begin{proposition}\label{prop: explicit formula for diag val}
For $\bm{\ell}\in\NN_0^{\frac{r(r+1)}{2}}$, we have 
\begin{align*}
 \zeta_{\A_r}\overset{\diag}{(-\bm{\ell})}
 = \prod_{1\le i<j \le r}
   \sum_{\substack{ k^{(ij)}_i+\cdots+k^{(ij)}_{j}=\ell_{ij} \\[1mm] k^{(ij)}_p\ge0 }} \!\!\!\!
   \binom{\ell_{ij}}{k^{(ij)}_i,\dots,k^{(ij)}_{j}} 
   \prod_{i=1}^r(-1)^{\ell_{ii}+K_i}\frac{B_{\ell_{ii}+K_i+1}}{\ell_{ii}+K_i+1}. 
\end{align*}
\end{proposition}

\begin{proof}
 This is a direct consequence of Lemma \ref{decomp formula for the reg zetaAr}.
\end{proof}

The vanishing of the diagonal values with the tuple $\bm{\ell}\in\NN_{0}^{\frac{r(r+1)}{2}}$ satisfying the trivial zero condition immediately follows from this explicit formula. 

\begin{theorem}\label{thm: vanishing of diag val}
Let $\bm{\ell}\in\NN_0^{\frac{r(r+1)}{2}}$ with $\ell_{ii}>0$ ($i=1,\dots,r$) and  $\wt(\bm{\ell})\equiv r-1\mod 2$.
Then
\begin{align*}
 \zeta_{\A_r}\overset{\diag}{(-\bm{\ell})}=0.
\end{align*} 
\end{theorem}

\begin{proof}
By the definition of $K_i$ (see \eqref{eqn: definition of Ki}), we have 
\begin{align*}
    \sum_{i=1}^r (\ell_{ii}+K_i+1) 
    &= 
    \sum_{i=1}^r \ell_{ii}+ \sum_{i=1}^r\sum_{\substack{1\le p\le i\le q\le r \\ p\ne q}} k_i^{(p\,q)} +r \\
    &= \wt(\bm{\ell})+r \equiv r-1+r \equiv1 \mod2.
\end{align*}
Thus, there exists $i=1,\dots,r$ such that $\ell_{ii}+K_i+1\equiv1\mod2$ and since $\ell_{ii}>0$, we have 
\[
   B_{\ell_{ii}+K_i+1}=0
\]
for such $i$. 
\end{proof}

Next, we calculate the reverse values.
Due to the following proposition, the reverse values are well-defined.

\begin{proposition}\label{prop: explicit formula for rev val}
For $\bm{\ell}\in\NN_0^{\frac{r(r+1)}{2}}$, we have 
\begin{align}\label{eqn: explicit formula for rev. val. of zeta_Ar}
\zeta_{\A_r}(-\bm{\ell};s_{1r})
= \prod_{2\le i\le j\le r} \sum_{\substack{p_{ij}+q_{ij}=\ell_{ij} \\ p_{ij},q_{ij}\ge0}}\!\!
    \binom{\ell_{ij}}{q_{ij}} (-1)^{q_{ij}}
    \zeta_{r}(-\bm{\ell}'(\bm{p},\bm{q}),s_{1r}(\bm{p})),
\end{align}
where $\zeta_{r}(-\bm{\ell}'(\bm{p},\bm{q}),s_{1r}(\bm{p}))$ is the Euler--Zagier multiple zeta function and its $k$-th component is 
\begin{align*}
-\ell_{1j}-\sum_{i=2}^{k}p_{ik} -\sum_{j=k+1}^{r}q_{k+1\,j} \quad &(k=1,\dots,r-1),\\
s_{1r}-\sum_{i=2}^{r}p_{ir}\quad &(k=r).
\end{align*}
The empty summation is interpreted as $0$. 
In particular, we have 
\begin{equation}\label{eqn: expression of ar-reverse in terms of EZ-reverse}
\zeta_{\A_r}\overset{\rev}{(-\bm{\ell})}
= \sum_{2\le i\le j\le r}\ \sum_{\substack{p_{ij}+q_{ij}=\ell_{ij} \\ p_{ij},q_{ij}\ge0}}\ 
   \binom{\ell_{ij}}{p_{ij}} (-1)^{q_{ij}} \zeta_{r}^{\rev}(-\bm{\ell}'(\bm{p},\bm{q}),-\ell_{1r}(\bm{p})).
\end{equation}
\end{proposition}

\begin{proof}
For \eqref{eqn: explicit formula for rev. val. of zeta_Ar}, 
just use the binomial theorem for the $(i,j)$-th component ($2\le i\le j\le r$) of the summand of $\zeta_{\A_r}(-\bm{\ell};s_{1r})$. 
Taking the limit $s_{1r}\rightarrow-\ell_{1r}$, we obtain the formula \eqref{eqn: expression of ar-reverse in terms of EZ-reverse}.
\end{proof}
\section{Trivial zeros of zeta functions of type \texorpdfstring{$\A_r$}{Ar}}
In this section, we prove our main Theorem \ref{thm: main theorem}.
In \S \ref{subsec: trivial zeros of reg}, we shall prove the vanishing of regular values (Theorem \ref{thm: trivial zeros of reg Ar zeta}.
To prove it, we consider the generating series of regular values of $\zeta_{\A_r}(\bm{s})$ at non-positive integers.
In \S \ref{subsec: Trivial zeros of reverse type}, we prove the vanishing of reverse values with a different method (Theorem \ref{th:vanishing_reverse_values}).

\subsection{Trivial zeros of regular values}\label{subsec: trivial zeros of reg}
In this subsection, we prove Theorem \ref{thm: trivial zeros of reg Ar zeta} and Theorem \ref{thm: vanishing of diag val}.   
For our proof, we consider a generating series of the regular values $\zeta_{\A_r}\overset{\reg}{(-\bm{\ell})}$ and the diagonal values $\zeta_{\A_r}\overset{\diag}{(-\bm{\ell})}$. 
We begin with a more general function;
\begin{equation*}
 Z_{\A_r}(\bm{s};\bm{t})
 := \sum_{m_1,\dots,m_r\in\NN} 
    \prod_{1\le i\le j\le r}
    \frac{ t_{ij}^{m_i+\cdots+m_j} }
         { (m_i+\cdots+m_j)^{s_{ij}} }
\end{equation*}
with $t_{ij}\in\CC$ such that $|t_{ij}|<1$.
Note that $Z_{\A_r}(\bm{s};\bm{t})$ converges for all $\bm{s}\in\CC^{\frac{r(r+1)}{2}}$.
We next put
\begin{equation*}
 G_{\A_r}(\bm{X}):=Z_{\A_r}(\bm{0};(\ee^{X_{ij}})).
\end{equation*}

The following lemma is fundamental.
\begin{lemma}\label{lem: Explicit formula for GAr}
We have 
\begin{align*}
  G_{\A_r}(\bm{X}) = \prod_{j=1}^r 
                       \frac{ \exp(\sum_{1\le k\le j\le \ell\le r}X_{k\ell}) }
                            { 1 - \exp(\sum_{1\le k\le j\le \ell\le r}X_{k\ell}) }.
\end{align*}
\end{lemma}

\begin{proof}
Since 
\begin{align*}
Z_{\A_r}(\bm{0};(\ee^{X_{ij}}))
&= \prod_{j=1}^r \sum_{m_j\in\NN}
   \exp\left(m_j\left(\sum_{1\le k\le j\le \ell\le r} X_{k\ell}\right) \right),
\end{align*}
we get the claim.
\end{proof}

By definition, we find that the function $G_{\A_r}(\bm{X})$ has a singularity at $\bm{X}=\bm{0}$, 
but using the recurrence formula of regular values \eqref{eqn: rec. rel. for reg Ar zeta}, we can obtain a power series.  
We now explain this \emph{regularization procedure of $G_{\A_r}(\bm{X})$}.
Note that when $r=1$, $G_{\A_1}(X)$ is essentially the generating function of $\zeta(-\ell)$ ($\ell \in\NN_0$). Indeed, we have 
\begin{align*}
 G_{\A_1}(X)
 = \frac{1}{\ee^{-X}-1}
 = -\frac{1}{X} + \sum_{\ell\ge0} \zeta(-\ell)\frac{X^\ell}{\ell!}.
\end{align*}
Therefore, the function $G^{\rm reg}_{\A_1}(X)$ defined by
\begin{equation}\label{eqn: GA1}
 G^{\rm reg}_{\A_1}(X) := G_{\A_1}(X) + \frac{1}{X} 
 = \sum_{\ell\ge0} \zeta(-\ell)\frac{X^\ell}{\ell!}
\end{equation}
is nothing but the generating series of $\zeta(-\ell)$ ($\ell \in\NN_0$).
From this point of view and from the recurrence formula\eqref{eqn: rec. rel. for reg Ar zeta} of $\zeta_{\A_r}\overset{\reg}{(-\bm{\ell})}$, we consider the following.

\begin{definition}
Let $r\ge2$. 
For variables $\bm{X}=(X_{ij})_{1\le i\le j\le r}$, we define
\begin{equation}\label{eqn: rec rel for gen func of reg}
 G^{\rm reg}_{\A_r}(\bm{X})
 := G^{\rm reg}_{\A_{r-1}}(\bm{X}')G^{\rm reg}_{\A_1}(X'')
    + \frac{ G^{\rm reg}_{\A_{r-1}}(\bm{X}'-X'')-G^{\rm reg}_{\A_{r-1}}(\bm{X}') }{X''}, 
\end{equation}
where 
\begin{align*}
 \bm{X}'=(X'_{ij})_{1\le i\le j\le r-1} 
 := \begin{cases}
      X_{ij} \quad (1\le i\le j\le r-2), \\
      X_{i\,r-1}+X_{i\,r} \quad (1\le i\le j=r-1),
    \end{cases} 
\end{align*}
and $X'':=\sum_{j=1}^rX_{jr}$. We also set 
\begin{align*}
 \bm{X}'-X'' 
 := \begin{cases}
      X_{ij} \quad &(1\le i\le j\le r-2), \\
      X_{1\,r-1}+X_{1r}-X'' \quad &(i=1,\,j=r-1),\\
      X_{i\,r-1}+X_{i\,r} \quad &(2\le i\le j=r-1).
    \end{cases} 
\end{align*}
\end{definition}

The function $G_{\A_r}^{\rm reg}(\bm{X})$ is the generating series of regular values. 
To see this, we first prove that $G^{\rm reg}_{\A_r}(\bm{X})$ is a power series.

\begin{lemma}
We have 
\begin{align*}
  G^{\rm reg}_{\A_r}(\bm{X})\in\CC[[\,(X_{ij})_{1\le i\le j\le r}\,]].
\end{align*} 
\end{lemma}

\begin{proof}
We prove the claim by induction on $r$.
For $r=1$, this case is clear from \eqref{eqn: GA1}.
For $r>1$, we assume that $G^{\rm reg}_{\A_j}$ are the power series for any $j=1,\dots,r-1$.
From \eqref{eqn: rec rel for gen func of reg}, 
since the induction hypothesis implies that 
\begin{align*}
G^{\rm reg}_{\A_{r-1}}(\bm{X}')G^{\rm reg}_{\A_1}(X'')
\end{align*}
forms a power series, 
it is enough to show that 
\begin{align*}
\frac{ G^{\rm reg}_{\A_{r-1}}(\bm{X}'-X'')-G^{\rm reg}_{\A_{r-1}}(\bm{X}') }{X''}
\end{align*}
is a power series. Here, recall that
\begin{align*}
 \bm{X}'=(X'_{ij})_{1\le i\le j\le r-1} 
 := \begin{cases}
      X_{ij} \quad (1\le i\le j\le r-2), \\
      X_{i\,r-1}+X_{i\,r} \quad (1\le i\le j=r-1),
    \end{cases} 
\end{align*}
$X'':=\sum_{j=1}^rX_{jr}$, and 
\begin{align*}
 \bm{X}'-X'' 
 := \begin{cases}
      X_{ij} \quad &(1\le i\le j\le r-2), \\
      X_{1\,r-1}+X_{1r}-X'' \quad &(i=1,\,j=r-1),\\
      X_{i\,r-1}+X_{i\,r} \quad &(2\le i\le j=r-1).
    \end{cases} 
\end{align*}
By the induction hypothesis, we know that
\begin{align*}
 G^{\rm reg}_{\A_{r-1}}(\bm{Y})
 =G^{\rm reg}_{\A_{r-1}}((Y_{ij})_{1\le i\le j\le r-1})
 \in \CC[[(Y_{ij})_{1\le i\le j\le r-1}]].
\end{align*}
Thus, we can write 
\begin{align*}
 G^{\rm reg}_{\A_{r-1}}(\bm{Y})
 = \sum_{\substack{ \ell_{ij}\ge0 \\ 1\le i\le j\le r-1 }}
     c(\ell_{ij}) \frac{Y_{ij}^{\ell_{ij}}}{\ell_{ij}!}
\end{align*}
with some coefficients $c(\ell_{ij})\in\CC$. 
By this expression, we have
{\small
\begin{align*}
&G^{\rm reg}_{\A_{r-1}}(\bm{X}'-X'')-G^{\rm reg}_{\A_{r-1}}(\bm{X}') \\
&= \sum_{\substack{ \ell_{ij}\ge0 \\ 1\le i\le j\le r-1 }}
     c(\ell_{ij}) 
     \left(
     \prod_{ \substack{ 1\le i\le j \le r-1 \\  (i,j)\ne(1,r-1) }}\frac{X_{ij}^{\ell_{ij}}}{\ell_{ij}!}\right)\cdot
     \left[
       \frac{ (X_{1\,r-1}+X_{1r}-X'')^{\ell_{1\,r-1}} - (X_{1\,r-1}+X_{1r})^{\ell_{1\,r-1}}  }
            {\ell_{1\,r-1}!}\right]
\end{align*}}
It is obvious that for any $\ell_{1\,r-1}\ge0$, 
\begin{align*}
 (X_{1\,r-1}+X_{1r}-X'')^{\ell_{1\,r-1}} - (X_{1\,r-1}+X_{1r})^{\ell_{1\,r-1}} 
\end{align*}
is divisible by $X''=\sum_{j=1}^rX_{jr}$.
We therefore obtain the claim.
\end{proof}

Now, we show that $G_{\A_r}^{\rm reg}(\bm{X})$ is the generating series of $\zeta_{\A_r}\overset{\reg}{(-\bm{\ell})}$.
rWe define 
\begin{align*}
  H_{\A_r}(\bm{X})
  := \sum_{\substack{ \ell_{ij}\ge0 \\ 1\le i\le j\le r}}
       \zeta_{\A_r}\overset{\reg}{(-\bm{\ell})}
       \frac{ X_{ij}^{\ell_{ij}} }{\ell_{ij}!} \in \CC[[\,(X_{ij})\,]].
\end{align*}

\begin{proposition}
We have 
\begin{align*}
  H_{\A_r}(\bm{X})=G^{\rm reg}_{\A_r}(\bm{X}).
\end{align*}
\end{proposition}

\begin{proof}
When $r=1$, we have 
\begin{align*}
  H_{\A_1}(\bm{X}) = \sum_{\ell\ge0}\zeta(-\ell)\frac{X^\ell}{\ell!} = G^{\rm reg}_{\A_1}(X).
\end{align*}
When $r>1$, from \eqref{eqn: rec. rel. for reg Ar zeta}, 
we have 
\begin{align*}
H_{\A_r}(\bm{X})
= \sum_{\substack{ \ell_{ij}\ge0 \\ 1\le i\le j\le r}}
  \zeta_{\A_r}\overset{\reg}{(-\bm{\ell})}
  \frac{ X_{ij}^{\ell_{ij}} }{\ell_{ij}!} 
= F_1(\bm{X}) + F_2(\bm{X}),
\end{align*}
where 
\begin{align*}
 F_1(\bm{X})
 := \sum_{\substack{ \ell_{ij}\ge0 \\ 1\le i\le j\le r}}
  \left(
    \prod_{k=1}^{r-1}\ 
    \sum_{p_k+q_k=\ell_{kr}}
   \binom{\ell_{kr}}{q_k}
     \zeta_{\A_{r-1}}\overset{\reg}{(-\bm{\ell}'(\bm{p}))}
     \zeta(-\ell_{rr}-|\bm{q}|)
  \right)
  \frac{ X_{ij}^{\ell_{ij}} }{\ell_{ij}!}
\end{align*} 
and 
{\small
\begin{align*}
F_2(\bm{X}):=\!\!\!
\sum_{\substack{ \ell_{ij}\ge0 \\ 1\le i\le j\le r}}
  \left(
    \frac{ (-1)^{\ell_{rr}+|\bm{q}'|+1} }{\ell_{1r}+\ell_{rr}+|\bm{q}'|+1}
    \binom{\ell_{1r}+\ell_{rr}+|\bm{q}'|}{\ell_{1r}}^{-1}\!\!\!\cdot
    \prod_{k=2}^{r-1}\
    \sum_{p_k+q_k=\ell_{kr}}\!\!\!
    \binom{\ell_{kr}}{q_k}
    \zeta_{\A_{r-1}}\overset{\reg}{(-\bm{\ell}''(\bm{p}))}
  \right)
  \frac{ X_{ij}^{\ell_{ij}} }{\ell_{ij}!}.
\end{align*}}
Thus, it is enough to show
\begin{align}
  &F_1(\bm{X}) = H_{\A_{r-1}}(\bm{X}')H_{\A_1}(X'') \tag{i}, \\
  &F_2(\bm{X}) = \frac{ H_{\A_{r-1}}(\bm{X}'-X'') -H_{\A_{r-1}}(\bm{X}') }
                      { X'' }
                      \tag{ii}.
\end{align}
For point (i), since we have
\begin{align*}
  &(X_{i\,r-1}+X_{ir})^{\ell'_{i\,r-1}}
  = \sum_{p_i=0}^{\ell'_{i\,r-1}}
      \binom{\ell'_{i\,r-1}}{p_i} 
      X_{i\,r-1}^{\ell'_{i\,r-1}-p_i}X_{ir}^{p_i},\\
      &(X'')^{\ell}=(X_{1r}+X_{2r}+\cdots+X_{rr})^{\ell}
  = \sum_{\substack{q_1+\cdots+q_r=\ell \\ q_k\ge0}}
      \frac{\ell!}{q_1!\cdots q_r!}X_{1r}^{q_1}\cdots X_{rr}^{q_r},
\end{align*}
we can calculate 
\begin{align*}
 H_{\A_{r-1}}(\bm{X}')H_{\A_1}(X'')
 &= \left(
      \sum_{\substack{ \ell'_{ij}\ge0 \\ 1\le i\le j\le r-1}}
         \zeta_{\A_{r-1}}\overset{\reg}{(-\bm{\ell}')}
         \frac{ (X_{ij}')^{\ell'_{ij}} }{\ell'_{ij}!} \right)\cdot
          \left( \sum_{\ell\ge0}\zeta(-\ell)\frac{(X'')^\ell}{\ell!} \right) \\
 &=  \sum_{\substack{ \ell'_{ij}\ge0 \\ 1\le i\le j\le r-1}}
        \sum_{\substack{ q_k\ge0 \\ 1\le k\le r }}
        \sum_{p_i=0}^{\ell'_{i\,r-1}}
        \binom{\ell'_{i\,r-1}}{p_i}
        \zeta_{\A_{r-1}}\overset{\reg}{(-\bm{\ell}')}
        \zeta(-|\bm{q}|) \\
&\qquad \cdot 
  \left(
     \prod_{1\le i\le j\le r-2}
     \frac{ (X_{ij}')^{\ell'_{ij}} }{\ell'_{ij}!} \right)
     \left(
        \prod_{1\le i\le r-1}
        \frac{X_{i\,r-1}^{\ell'_{i\,r-1}-p_i}}{\ell'_{i\,r-1}!}
        \frac{X_{ir}^{p_i+q_i}}{q_{i}!} \right)
        \frac{X_{rr}^{q_r}}{q_r!}.
\end{align*}
Thus, we set 
\begin{align*}
 \bm{\ell}=(\ell_{ij})_{1\le i\le j\le r}
 = \begin{cases}
     \ell'_{ij} \quad &(1\le i\le j\le r-2), \\ 
     \ell'_{i\,r-1}-p_i \quad &(1\le i\le r-1,\, j=r-1), \\
     p_i+q_i \quad &(1\le i\le r-1,\, j=r),\\
     q_{r} \quad &(i=j=r),
   \end{cases}
\end{align*}
and get 
{\small
\begin{align*}
H_{\A_{r-1}}(\bm{X}')H_{\A_1}(X'')
= \sum_{\substack{ \ell_{ij}\ge0 \\ 1\le i\le j\le r}}
  \left(
    \prod_{k=1}^{r-1}\ 
    \sum_{p_k+q_k=\ell_{kr}}
       \binom{\ell_{kr}}{q_k}
       \zeta_{\A_{r-1}}\overset{\reg}{(-\bm{\ell}'(\bm{p}))}
       \zeta(-\ell_{rr}-|\bm{q}|)
  \right)
  \frac{ X_{ij}^{\ell_{ij}} }{\ell_{ij}!}.
\end{align*}}
The right-hand side is equal to $F_1(\bm{X})$.

For point (ii), since 
\begin{align*}
 &(X_{1\,r-1}+X_{1r}-X'')^{\ell_{1\,r-1}} - (X_{1\,r-1}+X_{1r})^{\ell_{1\,r-1}} \\
 &= \sum_{k=1}^{\ell_{1\,r-1}}
      \binom{\ell_{1\,r-1}}{k}(X_{1\,r-1}+X_{1r})^{\ell_{1\,r-1}-k}(-X'')^k,
\end{align*}
we have 
\begin{align*}
&\frac{ H_{\A_{r-1}}(\bm{X}'-X'') -H_{\A_{r-1}}(\bm{X}') }
                      {X''}\\
&= \sum_{\substack{ \ell_{ij}\ge0 \\ 1\le i\le j\le r-1}}
      \zeta_{\A_{r-1}}\overset{\reg}{(-\bm{\ell})}
      \left(
         \prod_{1\le i\le j\le r-2}
         \frac{ X_{ij}^{\ell_{ij}} }{\ell_{ij}!} 
      \right)
      \left(
        A\cdot
        \sum_{\substack{ 0\le p_i\le \ell_{i\,r-1} \\ 2\le i\le r-1 }}
          \binom{p_i+q_i}{p_i}
      \right)
      \\
&\qquad\cdot  
   \frac{X_{1\,r-1}^{\ell_{1\,r-1}-k-p}}{(\ell_{1\,r-1}-k-p)!} 
   \left(
       \sum_{\substack{ 0\le p_i\le \ell_{i\,r-1} \\ 2\le i\le r-1 }}
      \frac{ X_{i\,r-1}^{\ell_{i\,r-1}-p_i} }{ (\ell_{i\,r-1}-p_i)! }
   \right)
   \frac{X_{1r}^{p+q_1}}{(p+q_1)!}
   \left(
      \prod_{2\le i\le r-1}
      \frac{ X_{ir}^{p_i+q_i} }{ (p_i+q_i)! }
   \right)
   \frac{X_{rr}^{q_r}}{q_r!},
\end{align*}
where
\begin{align*}
A := \sum_{k=1}^{\ell_{1\,r-1}}
        \sum_{p=0}^{\ell_{1\,r-1}-k}\!\!
        \sum_{\substack{q_1+\cdots+q_r=k-1 \\ q_j\ge0}}\!\!
           \frac{ (-1)^{k} (k-1)! }{k!}
           \binom{p+q_1}{p}.
\end{align*}
Thus, by putting 
\begin{align*}
 \bm{m}=(m_{ij})_{1\le i\le j\le r}
 = \begin{cases}
     \ell_{ij}                       \quad &(1\le i\le j\le r-2), \\ 
     \ell_{1\,r-1}-|\bm{q}|-1-p      \quad &(i=1,\, j=r-1), \\
     p+q_1                        \quad &(i=1,\, j=r), \\ 
     \ell_{i\,r-1}-p_i                      \quad &(2\le i\le j=r-1),\\
     p_i+q_i           \quad &(2\le i\le r-1,\, j=r),\\
     q_{r}                        \quad &(i=j=r),
   \end{cases}
\end{align*}
and using the following fundamental identity
\begin{align*}
\sum_{q_1=0}^{m_{1r}}
  \binom{m_{1r}}{q_1}\frac{(-1)^{q_1}}{q_1+|\bm{q}'|+1}
= \frac{1}{m_{1r}+|\bm{q}'|+1}\binom{m_{1r}+|\bm{q}'|}{m_{1r}}^{-1}, 
\end{align*}
we obtain point (ii).

Since both $H_{\A_r}(\bm{X})$ and $G^{\rm reg}_{\A_r}(\bm{X})$ have the same recurrence formulas and coincide with each other in the initial case $r=1$, we obtain the result.
\end{proof}

Using the function $G_{\A_r}^{\reg}(\bm{X})$, we prove the vanishing of regular values with the tuple satisfying the trivial zero condition \eqref{eq:trivial_zero_condition}.

\begin{theorem}\label{thm: trivial zeros of reg Ar zeta}
Let $\bm{\ell}\in\NN_0^{\frac{r(r+1)}{2}}$ with $\ell_{ii}>0$ ($i=1,\dots,r$) and  $\wt(\bm{\ell})\equiv r-1\mod 2$. Then 
\begin{align*}
 \zeta_{\A_r}\overset{\reg}{(-\bm{\ell})}=0.
\end{align*} 
\end{theorem}

\begin{proof}
Put
\begin{align*}
   G^{\rm reg\,+}_{\A_r}(\bm{X}) 
   := \sum_{\ell_{11},\dots,\ell_{rr}>0}
      \sum_{\substack{\ell_{ij}\ge0 \\ 1\le i< j\le r}}
         \zeta_{\A_r}\overset{\reg}{(-\bm{\ell})} \prod_{1\le i\le j\le r} \frac{X_{ij}^{\ell_{ij}}}{\ell_{ij}!}.
\end{align*}
Then, Theorem \ref{thm: trivial zeros of reg Ar zeta} is equivalent to the following claim.\\[5mm]
\textbf{Claim.}
{\it The function $G^{\rm reg\,+}_{\A_r}(X)$ is an odd function when $r\equiv1 \mod2$ and an even function when $r\equiv0 \mod2$.}\\[5mm]
We prove this claim by induction on $r$.
When $r=1$, since 
\begin{align*}
G^{\rm reg\,+}_{\A_1}(X)
:= G^{\rm reg}_{\A_1}(X)-\zeta(0)
= \frac{1}{\ee^{-X}-1}+\frac{1}{X}+\frac{1}{2}
\left(=\sum_{\ell\ge1}\zeta(-\ell)\frac{X^\ell}{\ell!}\right)
\end{align*}
is an odd function, and thus, the claim is clear.

We move on to the case $r>1$. 
By \eqref{eqn: rec rel for gen func of reg}, we have  
\begin{align*}
G^{\rm reg}_{\A_{r}}
= G^{\rm reg}_{\A_{r-1}}(\bm{X}')G^{\rm reg}_{\A_1}(X'')
    + \frac{ G^{\rm reg}_{\A_{r-1}}(\bm{X}'-X'')-G^{\rm reg}_{\A_{r-1}}(\bm{X}') }{X''}. 
\end{align*}
By the induction hypothesis, 
\begin{align*}
G^{\rm reg\,+}_{\A_{r-1}}(\bm{Y})
:= \sum_{\substack{\ell_{ii}>0 \\ 1\le i\le r-1}}\,
   \sum_{\substack{\ell_{ij}\ge0 \\ 1\le i< j\le r-1}}
     \zeta_{\A_{r-1}}\overset{\reg}{(-\bm{\ell})}
     \prod_{1\le i\le j\le r-1}
     \frac{Y_{ij}^{\ell_{ij}}}{\ell_{ij}!}
\end{align*}
is an odd function for $r\equiv0 \mod2$ and an even function for $r\equiv1 \mod2$. 
Since one can prove the claim in a similar manner, we assume $r\equiv0 \mod2$.
We set 
\begin{align*}
 g_{\A_{r-1}}^{\rm reg}(\bm{Y})
 := G^{\rm reg}_{\A_{r-1}}(\bm{Y}) - G^{\rm reg\,+}_{\A_{r-1}}(\bm{Y})
 \in\CC[[\bm{Y}]].
\end{align*}
By definition, $g_{\A_{r-1}}^{\rm reg}(\bm{Y})$ is the sum of all constant functions $G^{\rm reg}_{\A_{r-1}}(\bm{Y})\Big| {}_{Y_{ii}=0}$ for some $1\le i\le r-1$.
Note that one can write
\begin{equation}\label{eqn: decomp of Gar-1Ga1}\begin{split}
&G^{\rm reg}_{\A_{r-1}}(\bm{X}')G^{\rm reg}_{\A_1}(X'')
= (G^{\rm reg +}_{\A_{r-1}}(\bm{X}') + g_{\A_{r-1}}^{\rm reg}(\bm{X}'))
   (G^{\rm reg+}_{\A_1}(X'') + \zeta(0))\\
&= G^{\rm reg +}_{\A_{r-1}}(\bm{X}')G^{\rm reg+}_{\A_1}(X'')
   + G^{\rm reg +}_{\A_{r-1}}(\bm{X}')\zeta(0) 
   + g_{\A_{r-1}}^{\rm reg}(\bm{X}')G^{\rm reg+}_{\A_1}(X'')
   + g_{\A_{r-1}}^{\rm reg}(\bm{X}')\zeta(0)
\end{split}\end{equation}
and the values $\zeta_{\A_{r}}\overset{\reg}{(-\bm{\ell})}$ with $\ell_{ii}>0$ ($1\le i\le r$) and $\ell_{ij}\ge0$ ($1\le i< j\le r$) come from 
\begin{align*}
G^{\rm reg +}_{\A_{r-1}}(\bm{X}')G^{\rm reg+}_{\A_1}(X''), \quad
\frac{ G^{\rm reg}_{\A_{r-1}}(\bm{X}'-X'')-G^{\rm reg}_{\A_{r-1}}(\bm{X}') }{X''}.
\end{align*}
Because the later three parts on the right-hand side of \eqref{eqn: decomp of Gar-1Ga1} are the constant function in some $X_{ii}$ ($1\le i\le r$).
In other words, coefficients of the sum 
\begin{align*}
G^{\rm reg +}_{\A_{r-1}}(\bm{X}')\zeta(0) 
   + g_{\A_{r-1}}^{\rm reg}(\bm{X}')G^{\rm reg+}_{\A_1}(X'')
   + g_{\A_{r-1}}^{\rm reg}(\bm{X}')\zeta(0)
\end{align*}
are the regular values of the form
\begin{align*}
 \zeta_{\A_r}\overset{\reg}{(-\bm{\ell})}
 \quad (\bm{\ell}=(\ell_{ij}) \in \NN_0^{\frac{r(r+1)}{2}}, \ \ell_{ii}=0 \ \ \text{for some } 1\le i\le r).
\end{align*}

The induction hypothesis implies that 
\begin{align*}
G^{\rm reg +}_{\A_{r-1}}(\bm{X}')G^{\rm reg+}_{\A_1}(X''), \quad
\frac{ G^{\rm reg +}_{\A_{r-1}}(\bm{X}'-X'')-G^{\rm reg +}_{\A_{r-1}}(\bm{X}') }{X''}
\end{align*}
are even functions, and thus we obtain the claim.
\end{proof}

\begin{remark}
We cannot remove the condition $\ell_{ii}>0$ ($1\le i\le r$). 
Indeed, under the condition $\ell_{11}+\ell_{12}\equiv1\mod2$, we have 
\begin{align*}
    \zeta_{\A_2}\overset{\reg}{(\begin{smallmatrix} -\ell_{11}&-\ell_{12}\\&0 \end{smallmatrix})}
    = \zeta_2^{\reg}(-\ell_{11},-\ell_{12}) = \zeta(-\ell_{11}-\ell_{12})\zeta(0).
\end{align*}
However, we know $\zeta(-\ell_{11}-\ell_{12})\zeta(0)\ne0$.
\end{remark}

\subsection{Trivial zeros of reverse type}\label{subsec: Trivial zeros of reverse type}
In this section, we prove the vanishing of the reverse values with a different method. We consider a Hurwitz-type zeta function and prove some symmetry of parameters (see Theorem \ref{th:vanishing_reverse_values}).

Let $r>1$, and let $\bm{\ell} = (\ell_{ij})_{1 \leq i \leq j \leq r}\in\NN_0^{\frac{r(r+1)}{2}}$ and $\bm{x}=(x_1,\ldots,x_r)$ be positive real numbers. We set
\[
    Z_{\bm{\ell}}(s;\bm{x}) := \sum_{m_1,\ldots,m_r \geq 0} \frac{\prod_{\substack{1 \leq i \leq j \leq r \\ (i,j) \neq (1,r)}} (m_i+x_i+\cdots+m_j+x_j)^{\ell_{ij}}}{(m_1+x_1+\cdots+m_r+x_r)^s}.
\]
When $\bm{x} = \bm{1}$, we observe that the values at $s=-\ell_{1r}$ correspond to the reverse values $\zeta_{\A_r}(\overset{\rev}{-\bm{\ell}})$.

\begin{theorem}
    \label{th:vanishing_reverse_values}
    We have
    \[
        Z_{\bm{\ell}}(-\ell_{1r};\bm{x}) = (-1)^{\wt(\bm{\ell})-r} Z_{\bm{\ell}}(-\ell_{1r};\bm{1}-\bm{x}) \qquad (0<x_1,\ldots,x_r<1)
    \]
    If we further assume the trivial zero condition \eqref{eq:trivial_zero_condition}, then we have
    \[
        \zeta_{\A_r}(\overset{\rev}{-\bm{\ell}}) = 0.
    \]
\end{theorem}

Because of the symmetry $(x_1,\ldots,x_r) \leftrightarrow (1-x_{1},\ldots,1-x_r)$, the limit $\lim_{\bm{x} \to \bm{0}} Z_{\bm{\ell}}(-\ell_{1r};\bm{x})$ is well-defined and shall be denoted by $Z_{\bm{\ell}}(-\ell_{1r};\bm{0})$. We show the following lemma.

\begin{lemma}
    \label{lemma:symmetry_barnes}
    Let $\bm{n}=(n_1,\ldots,n_r)$ be non-negative integers and set
    \[
        F_{\bm{n}}(s;\bm{x}) := \sum_{m_1,\ldots,m_r \geq 0} \frac{(m_1+x_1)^{n_1} \cdots (m_r+x_r)^{n_r}}{(m_1+x_1+\cdots+m_r+x_r)^s} \qquad (\re(s) \gg 1).
    \]
    Then we have
    \[
        F_{\bm{n}}(-m;\bm{1}-\bm{x}) = (-1)^{m+n_1+\cdots+n_r-r}F_{\bm{n}}(-m;\bm{x}) \qquad (m \in \NN_{0}),
    \]
    and denote $F_{\bm{n}}(-m;\bm{0}) := \lim_{\bm{x} \to \bm{0}} F_{\bm{n}}(-m;\bm{x})$. Furthermore, when $n_1,\ldots,n_r \geq 1$ we have
    \[
        F_{\bm{n}}(-m;\bm{0}) = F_{\bm{n}}(-m;\bm{1}) \qquad (m \in \NN_{0}).
    \]
\end{lemma}

\begin{proof}
    By direct application of \cite[Theorem 20]{rutard2026values}, we find that there exist rational numbers $c_{\bm{k},\mathcal{A}}$ such that
    \begin{multline*}
        F_{\bm{n}}(-m;\bm{x}) \\
        = \sum_{\emptyset \neq \mathcal{A} \subseteq \llbracket 1,r \rrbracket} (-1)^{\beta + n_{b_1}+\cdots+n_{b_\beta}} \prod_{p=1}^{\beta} n_{b_p}! \sum_{\substack{k_1+\cdots+k_{\alpha} = \beta + m \\
        \quad +n_{b_1}+\cdots+n_{b_\beta}}} c_{\bm{k},\mathcal{A}} \prod_{p=1}^{\alpha} \frac{\zeta(-n_{a_p}-k_p;x_{a_p})}{k_p!}
    \end{multline*}
    where the first sum symbol corresponds to the sum over all nonempty sets $\mathcal{A}=\{a_1,\ldots,a_{\alpha}\} \subseteq \llbracket 1,r \rrbracket$, and where we denote $\mathcal{B} := \llbracket 1,r \rrbracket \setminus \mathcal{A} = \{ b_1,\ldots,b_{\beta} \}$.
    
     By replacing $(x_1,\ldots,x_r)$ by $(1-x_1,\ldots,1-x_r)$ and using the following symmetry property of the Hurwitz zeta function
    \[
        \zeta(-k;1-x) = (-1)^{k+1} \zeta(-k;x) \qquad (k \in \NN_{0}),
    \]
    we get
    \begin{multline*}
        F_{\bm{n}}(-m;\bm{1}-\bm{x}) = \sum_{\emptyset \neq \mathcal{A} \subseteq \llbracket 1,r \rrbracket} (-1)^{\beta+n_{b_1}+\cdots+n_{b_\beta}} \prod_{p=1}^{\beta} n_{b_p}! \\
        \times \sum_{\substack{k_1+\cdots+k_{\alpha} = \beta + m \\
        \quad +n_{b_1}+\cdots+n_{b_\beta}}} (-1)^{\alpha + k_1 + n_{a_1}+\cdots+k_{\alpha}+n_{a_{\alpha}}} c_{\bm{k},\mathcal{A}} \prod_{p=1}^{\alpha} \frac{\zeta(-n_{a_p}-k_p;x_{a_p})}{k_p!}.
    \end{multline*}
    Finally, we observe that $(-1)^{\alpha+k_1+n_{a_1}+\cdots+k_{\alpha}+n_{a_{\alpha}}} = (-1)^{m+n_1+\cdots+n_r+r}$, which gives $F_{\bm{n}}(-m;\bm{1}-\bm{x}) = (-1)^{m+n_1+\cdots+n_r+r} F_{\bm{n}}(-m;\bm{x})$. The last relation of the lemma follows from $\lim_{x \to 0} \zeta(-k;x) = \zeta(-k;1)$, which is true for all $k \in \mathbb{N}$.
\end{proof}

\begin{proof}[Proof of Theorem \ref{th:vanishing_reverse_values}]
The polynomial $\prod_{\substack{1 \leq i \leq j \leq r \\ (i,j) \neq (1,r)}} (X_i+\cdots+X_j)^{\ell_{ij}}$ is homogeneous of degree $\wt(\bm{\ell})-\ell_{1r}$ and divisible by $X_1^{\ell_{11}} X_2^{\ell_{22}} \cdots X_r^{\ell_{rr}}$. Therefore, we can write
\[
    \prod_{\substack{1 \leq i \leq j \leq r \\ (i,j) \neq (1,r)}} (X_i+\cdots+X_j)^{\ell_{ij}} = \sum_{\substack{n_1 \geq \ell_{11},\ldots,n_r \geq \ell_{rr} \\ n_1+\cdots+n_r = \wt(\bm{\ell})-\ell_{1r}}} C_{\bm{n}} X_1^{n_1} \cdots X_r^{n_r},
\]
where $C_{\bm{n}}$ are integer coefficients. 
We can then express $Z_{\bm{\ell}}(s;\bm{x})$ as a linear combination of $F_{\bm{n}}(s;\bm{x})$,
\[
    Z_{\bm{\ell}}(s;\bm{x}) = \sum_{\substack{n_1 \geq \ell_{11}, \ldots, n_r \geq \ell_{rr} \\ n_1+\cdots+n_r = \wt(\bm{\ell})-\ell_{1r}}} C_{\bm{n}} F_{\bm{n}}(\bm{s};\bm{x}).
\]
After applying Lemma \ref{lemma:symmetry_barnes}, we obtain $Z_{\bm{\ell}}(-\ell_{1r};\bm{x}) = (-1)^{\wt(\bm{\ell})-r} Z_{\bm{\ell}}(-\ell_{1r};\bm{1}-\bm{x})$.

Further assume that $\ell_{11},\ell_{22},\ldots,\ell_{rr} \geq 1$, then by applying Lemma \ref{lemma:symmetry_barnes}, we get that $F_{\bm{n}}(-\ell_{1r};\bm{0}) = F_{\bm{n}}(-\ell_{1r};\bm{1})$ for all $n_1 \geq \ell_{11}, \ldots, n_r \geq \ell_{rr}$. Therefore, $Z_{\bm{\ell}}(-\ell_{1r};\bm{0}) = Z_{\bm{\ell}}(-\ell_{1r};\bm{1})$. If we further assume that $\wt(\bm{\ell}) = r-1 \mod 2$, we get the vanishing result.
\end{proof}

\section{The other zeros of zeta functions of type \texorpdfstring{$A_r$}{A} and Eisenstein series}
\label{sec: other zeros}
In this section, we study the other zeros of the zeta functions of $\A_r$ at non-positive integer points. 
As stated in Introduction, Romik \cite{romik2017representations} and Au \cite{au2024vanishingwitten} proved the vanishing of the Witten zeta function at negative integer points. 
More precisely, Romik studied the zeta function $\zeta^W_{\A_2}(s):=\zeta_{\A_2}(s,s,s)$ and Au also obtained the vanishing results \cite[Theorem 1.2]{au2024vanishingwitten} for the general cases, including the other root systems. According to Au's results, for 
$\zeta_{\A_r}^W(s):=\zeta_{\A_r}(s,\dots,s)$, it follows 
\begin{align*}
  \zeta_{\A_r}^W(-\ell)=0 \quad (\ell\in\NN, \ r \geq 2).
\end{align*}

These studies naturally suggest a potential connection between the zeros of the ordered limit values at non-positive integer tuples that do not satisfy the trivial zero condition and the zeros of the Witten zeta functions.
We first treat the regular values.

\begin{proposition}\label{prop: reg Witten for A2}
Let $\ell\in\NN$ with $\ell\equiv0 \mod2$. We have 
\begin{align*}
   \zeta_{\A_2}\overset{\reg}{(\begin{smallmatrix} -\ell&-\ell\\&-\ell\end{smallmatrix})}=0.
\end{align*}
\end{proposition}

\begin{remark}
In this case, a tuple $(\ell,\ell,\ell)$ does not satisfy the trivial zero condition \eqref{eq:trivial_zero_condition} for $\zeta_{\A_2}(s_{11},s_{12},s_{22})$.
\end{remark}

\begin{proof}
By \eqref{eqn: example of reg rec rel for A2}, we have 
\begin{align*}
   \zeta_{\A_2} \overset{\reg}{\left(\begin{smallmatrix}-\ell&-\ell\\&-\ell\end{smallmatrix}\right)}
 = \sum_{k=0}^{\ell} \binom{\ell}{k} \zeta(-2\ell+k)\zeta(-\ell-k)
    - \frac{1}{2\ell+1} \binom{2\ell}{\ell}^{-1} \zeta(-3\ell-1).
\end{align*}
Note that we now assume $\ell>0$ is even, so the right-hand side can be written as 
\begin{equation}\label{eqn: shadow relation for A2 negative form}
 \sum_{k=1}^{n} \binom{2n}{2k-1} \zeta(-4n+2k-1)\zeta(-2n-2k-1)
    - \frac{1}{4n+1} \binom{4n}{2n}^{-1} \zeta(-6n-1).
\end{equation}
Here, we put $\ell=2n$. Romik \cite[(22)]{romik2017representations} has already proved that  \eqref{eqn: shadow relation for A2 negative form} is $0$ and therefore we obtain the claim.
\end{proof}

We review how to presume the identity among Eisenstein series from the zeros of $\zeta_{\A_r}$.
By the functional equation of the Riemann zeta function, we have 
\begin{align*}
    \zeta(-\ell) = \frac{\sin(\frac{\pi\ell}{2})}{2^{\ell}\pi^{\ell+1}}\ell!\zeta(\ell+1),
\end{align*}
and find that the vanishing $\zeta_{\A_2} \overset{\reg}{\left(\begin{smallmatrix}-\ell&-\ell\\&-\ell\end{smallmatrix}\right)}=0$ is equivalent to the following identity:
\begin{equation}\label{eqn: shadow relation for A2 positive form}
  \zeta(6n+2) 
    = \frac{2}{6n+1} \frac{(4n+1)!}{(2n)!^2} 
        \sum_{k=1}^{n} \frac{\binom{2n}{2k-1}}{\binom{6n}{2n+2k-1}} \zeta(4n-2k+2)\zeta(2n+2k).
\end{equation}
From the classical results in the theory of modular forms, we know that the Eisenstein series 
\begin{align*}
    G_k(\tau) = \sum_{(m,n)\in\ZZ^2\setminus\{(0,0)\}}\frac{1}{(m\tau+n)^k}
    = 2\zeta(k)\left( 1-\frac{2k}{B_k}\sum_{n\ge1}\sum_{d|n}d^{k-1}\ee^{2\pi \ii n\tau } \right),
\end{align*}
where $k\ge4$ is an even integer and $\tau\in\mathbb{H}$, has the Riemann zeta value as the constant term. 
Then, Romik speculated and proved that equation \eqref{eqn: shadow relation for A2 positive form} can be lifted to the identity among Eisenstein series:
\begin{equation}\label{eqn: Eis rel for A2}
  G_{6n+2}(\tau) 
  = \frac{1}{6n+1} \frac{(4n+1)!}{(2n)!^2} 
      \sum_{k=1}^{n} \frac{\binom{2n}{2k-1}}{\binom{6n}{2n+2k-1}} G_{4n-2k+2}(\tau)G_{2n+2k}(\tau).
\end{equation}
Note that we obtain \eqref{eqn: shadow relation for A2 positive form} by taking limit $\tau \rightarrow\ii\infty$ in \eqref{eqn: Eis rel for A2}.

Au also conjectured such non-trivial identities among $G_k(\tau)$ from the zeros of Witten zeta functions of other root systems.

It is then natural to wonder what happens in the general case. In particular, one might expect that 
\begin{align*}
    \zeta_{\A_r}\overset{\reg}
      {\left(\begin{smallmatrix} -\ell &\cdots&-\ell \\ 
                                  &\ddots&\vdots \\
                                  &      &-\ell  \end{smallmatrix}\right)}=0 
    \quad (\ell\in\NN \text{ with } \frac{r(r+1)}{2}\ell\equiv r\mod2 )
\end{align*}
However, this does not hold in general. We exhibit our observation as follows.
\begin{itemize}
\item For $\A_3$, since $6\ell\equiv0\mod2$ for any $\ell\in\NN$, all negative integer tuples $(-\ell,\dots-\ell)$ are the trivial zeros of regular values of $\zeta_{\A_3}$.
\item For $\A_4$, we conjecture the vanishing 
\begin{equation}\label{eqn: non trivial vanishing for A4}
    \zeta_{\A_4}\overset{\reg}
      {\left(\begin{smallmatrix} -\ell &-\ell &-\ell &-\ell \\ 
                                       &-\ell &-\ell &-\ell \\
                                       &      &-\ell &-\ell \\
                                       &      &      &-\ell \\\end{smallmatrix}\right)}=0 
\end{equation}
for any $\ell\in\NN$ by Mathematica calculation. Note that for any $\ell\in\NN$, since $10\ell\equiv0 \mod2$, all negative integer tuples $(-\ell,\dots-\ell)$ do not satisfy the trivial zero condition $\zeta_{\A_4}(\bm{s})$. 
\item For $\A_5$, in contrast to the cases $\A_2$ and $\A_4$, the first non-trivial case ($\ell=1$) gives
\begin{align*}
    \zeta_{\A_5}\overset{\reg}
      {\left(\begin{smallmatrix} -1 &-1&-1&-1&-1 \\ 
                                    &-1&-1&-1&-1 \\
                                    &  &-1&-1&-1 \\
                                    &  &  &-1&-1 \\
                                    &  &  &  &-1  \end{smallmatrix}\right)}
    =\frac{265469969}{2787700217702400}.
\end{align*}
\end{itemize}

For $A_r$ with $r\ge6$, we do not currently have a plausible conjecture.

If we assume the vanishing for $\zeta_{\A_4}\overset{\reg}{(-\ell,\dots,-\ell)}$, we conjecturally have the following identity among the Eisenstein series. Since we can obtain the following by just using \eqref{eqn: rec. rel. for reg Ar zeta} inductively and the result is quite long, we omit the details.
\begin{equation}\label{eqn: Eis rel for regA4}\begin{split}
   c_1G_{10\ell+4} 
   = &\sum_{k_2,k_3,k_4,k_6,k_7,k_9}\!\!\!\!\!
       c_2(\bm{k})G_{2\ell-k_9+k_{36}+1}G_{\ell+k_{479}+1}G_{4\ell-k_{234}+1}G_{3\ell-k_{67}+k_2+1} \\
     &\sum_{k_3,k_4,k_6,k_7,k_9} 
       c_3(\bm{k})G_{2\ell-k_9+k_{36}+1}G_{\ell+k_{479}+1}G_{7\ell-k_{3467}+2} \\
     &\sum_{k_2,k_4,k_6,k_7,k_9} 
       c_4(\bm{k})G_{\ell+k_{479}+1}G_{6\ell-k_{249}+k_6+2}G_{3\ell-k_{67}+k_2+1} \\
     &\sum_{k_4,k_6,k_7,k_9} 
       c_5(\bm{k})G_{\ell+k_{479}+1}G_{9\ell-k_{479}+3} \\
     &\sum_{k_2,k_3,k_6,k_7,k_9} 
       c_6(\bm{k})G_{2\ell-k_9+k_{36}+1}G_{5\ell+k_{79}-k_{23}+2}G_{3\ell-k_{67}+k_2+1} \\
     &\sum_{k_3,k_6,k_7,k_9} 
       c_7(\bm{k})G_{2\ell-k_9+k_{36}+1}G_{8\ell-k_{36}+k_9+3}\\
     &\sum_{k_2,k_6,k_7,k_9} 
       c_8(\bm{k})G_{7\ell+k_{67}-k_2+3}G_{3\ell-k_{67}+k_2+1}.
\end{split}\end{equation}
Here, each summation is finite.  We put $k_{pqrs}:=k_p+k_q+k_r+k_s$ for $p,q,r,s\in\{2,3,4,6,7,9\}$ ($k_{pq}$, $k_{pqr}$ are as well), $G_n:=G_n(\tau)$, and the coefficients $c_1, c_2(\bm{k}),\dots,c_8(\bm{k})\in\QQ$ depending on $\ell$ are described in Appendix below.
We should emphasize that we still do not know how to prove \eqref{eqn: non trivial vanishing for A4} and the identity \eqref{eqn: Eis rel for regA4}.

Lastly, we mention the other zeros of the diagonal and reverse values. We easily find that 
\begin{align*}
    \zeta_{\A_2}\overset{\bullet}{(\begin{smallmatrix} -\ell&-\ell\\&-\ell\end{smallmatrix})}\ne0
\end{align*}
for $\bullet\in\{\diag,\rev\}$ and even $\ell\in\NN$. Thus, we do not expect that we always have vanishing of any ordered limit values at non-positive integer points for general $\A_r$ ($r\ge3$), except for the case of trivial zeros.
It would be an interesting problem to see which ordered limit values vanish with negative integer tuples that do not satisfy the trivial zeros condition \eqref{eq:trivial_zero_condition}.

\appendix
\section{Complements}
As mentioned above, we provide an explanation of the computation of the coefficients $c_1, c_2(\bm{k}),\dots,c_8(\bm{k})\in\QQ$ appearing in \eqref{eqn: Eis rel for regA4}.
We apply \eqref{eqn: rec. rel. for reg Ar zeta} for $\zeta_{\A_4}\overset{\reg}{(-\bm{\ell})}$ inductively and obtain the following. 
{\small
\begin{align*}
&\zeta_{\A_4}\overset{\reg}
      {\left(\begin{smallmatrix} -\ell &-\ell &-\ell &-\ell \\ 
                                       &-\ell &-\ell &-\ell \\
                                       &      &-\ell &-\ell \\
                                       &      &      &-\ell \\\end{smallmatrix}\right)} \\
&= \sum_{k_4=0}^{\ell}\sum_{k_7=0}^{\ell}\sum_{k_9=0}^{\ell}
    \sum_{k_3=0}^{2\ell-k_4}\sum_{k_6=0}^{2\ell-k_7}\sum_{k_2=0}^{3\ell-k_{34}}
      \binom{\ell}{k_7}\binom{\ell}{k_9}\binom{\ell}{k_4}
      \binom{2\ell-k_7}{k_6}\binom{2\ell-k_4}{k_3}\binom{3\ell-k_{34}}{k_2}\\
&\hspace{45mm}\cdot    
    \zeta(-2\ell+k_9-k_{36} )\zeta(-\ell-k_{479})\zeta(-4\ell+k_{234})\zeta(-3\ell+k_{67}-k_2) \\
&\quad + 
  \sum_{k_4=0}^{\ell}\sum_{k_7=0}^{\ell}\sum_{k_9=0}^{\ell}\sum_{k_3=0}^{2\ell-k_4}\sum_{k_6=0}^{2\ell-k_7}
     \binom{\ell}{k_7}\binom{\ell}{k_9}\binom{\ell}{k_4}\binom{2\ell-k_4}{k_3}\binom{2\ell-k_7}{k_6}
     \frac{(-1)^{3\ell-k_{67}+1}}{6\ell-k_{3467}+1}
         \binom{6\ell-k_{3467}}{3\ell-k_{34}}^{-1}\\
&\hspace{45mm}\cdot
         \zeta(-2\ell+k_9-k_{36})\zeta(-\ell-k_{479})\zeta(-7\ell+k_{3467}-1) \\  
&\qquad+ 
   \sum_{k_4=0}^{\ell}\sum_{k_7=0}^{\ell}\sum_{k_9=0}^{\ell}
   \sum_{k_6=0}^{2\ell-k_7}\sum_{k_2=0}^{5\ell-k_{49}+k_6+1}
     \binom{\ell}{k_7}\binom{\ell}{k_9}\binom{\ell}{k_4}\binom{2\ell-k_7}{k_6}\binom{5\ell-k_{49}+k_6+1}{k_2}\\
&\hspace{15mm}\cdot     
     \frac{(-1)^{2\ell-k_9+k_6+1}}{4\ell-k_{49}+k_6+1}
     \binom{4\ell-k_{49}+k_6}{2\ell-k_4}^{-1}
       \zeta(-\ell-k_{479})\zeta(-6\ell+k_{249}-k_6-1)\zeta(-3\ell+k_{67}-k_2) \\
&\quad + 
   \sum_{k_4=0}^{\ell}\sum_{k_7=0}^{\ell}\sum_{k_9=0}^{\ell}\sum_{k_6=0}^{2\ell-k_7}
     \binom{\ell}{k_7}\binom{\ell}{k_9}\binom{\ell}{k_4}\binom{2\ell-k_7}{k_6}
     \frac{(-1)^{5\ell-k_{79}}}{(4\ell-k_{49}+k_6+1)(8\ell-k_{479}+2)} \\
&\hspace{35mm}\cdot
     \binom{4\ell-k_{49}+k_6}{2\ell-k_4}^{-1} \binom{8\ell-k_{479}+1}{5\ell-k_{49}+k_6+1}^{-1}
     \zeta(-\ell-k_{479})\zeta(-9\ell+k_{479}-2) \\ 
&\qquad + 
  \sum_{k_7=0}^{\ell}\sum_{k_9=0}^{\ell}
  \sum_{k_3=0}^{3\ell+k_{79}+1}\sum_{k_6=0}^{2\ell-k_7}\sum_{k_2=0}^{4\ell+k_{79}-k_3+1}
    \binom{\ell}{k_7}\binom{\ell}{k_9}\binom{2\ell-k_7}{k_6}\binom{3\ell+k_{79}+1}{k_3} 
    \binom{4\ell+k_{79}-k_3+1}{k_2}\\
&\hspace{15mm}\cdot 
    \frac{(-1)^{\ell+k_7+k_9+1}}{2\ell+k_{79}+1}
    \binom{2\ell+k_{79}}{\ell}^{-1} 
    \zeta(-2\ell+k_9-k_{36})\zeta(-5\ell-k_{79}+k_{23}-1)\zeta(-3\ell+k_{67}-k_2)\\
&\quad +
   \sum_{k_7=0}^{\ell}\sum_{k_9=0}^{\ell}\sum_{k_3=0}^{3\ell+k_{79}+1}\sum_{k_6=0}^{2\ell-k_7}
    \binom{\ell}{k_7}\binom{\ell}{k_9}\binom{2\ell-k_7}{k_6}\binom{3\ell+k_{79}+1}{k_3} 
    \frac{(-1)^{4\ell-k_{6}+k_9}}{(7\ell-k_{36}+k_9+2)(2\ell+k_{79}+1)}\\
&\hspace{23mm}\cdot
     \binom{7\ell-k_{36}+k_9+1}{4\ell+k_{79}-k_3+1}^{-1} \binom{2\ell+k_{79}}{\ell}^{-1}  
     \zeta(-2\ell+k_9-k_{36})\zeta(-8\ell+k_{36}-k_9-2)       \\
&\qquad+ 
  \sum_{k_7=0}^{\ell}\sum_{k_9=0}^{\ell}\sum_{k_6=0}^{2\ell-k_7}\sum_{k_2=0}^{6\ell+k_{67}+2}
    \binom{\ell}{k_7}\binom{\ell}{k_9}\binom{2\ell-k_7}{k_6}\binom{6\ell+k_{67}+2}{k_2} 
    \frac{(-1)^{3\ell+k_{67}}}{(2\ell+k_{79}+1)(5\ell+k_{67}+2)}\\
&\hspace{32mm}\cdot    
    \binom{2\ell+k_{79}}{\ell}^{-1}
    \binom{5\ell+k_{67}+1}{3\ell+k_{79}+1}^{-1}
      \zeta(-7\ell-k_{67}+k_2-2)\zeta(-3\ell+k_{67}-k_2)\\
&\quad + 
   \sum_{k_7=0}^{\ell}\sum_{k_9=0}^{\ell}\sum_{k_6=0}^{2\ell-k_7}
    \binom{\ell}{k_7}\binom{\ell}{k_9}\binom{2\ell-k_7}{k_6}
    \frac{ (-1)^{6\ell+3} }{(2\ell+k_{79}+1)(5\ell+k_{67}+2)(9\ell+3)}\\
&\hspace{32mm}\cdot  
    \binom{2\ell+k_{79}}{\ell}^{-1}
    \binom{5\ell+k_{67}+1}{3\ell+k_{79}+1}^{-1} 
    \binom{9\ell+2}{6\ell+k_{67}+2}^{-1}
         \zeta(-10\ell-3).
\end{align*}}
Combining this lengthy summation with the functional equation of the Riemann zeta function, we obtain the coefficients $c_1, c_2(\bm{k}),\dots,c_8(\bm{k})\in\QQ$.

\bibliography{./Bibliography.bib}
\bibliographystyle{alpha}
\end{document}